\documentclass[a4paper,10pt]{article}
\usepackage[utf8]{inputenc}
\usepackage{amsmath, amsfonts, amsthm,amssymb,  bbm, dsfont, lmodern}
\usepackage{math}
\usepackage{hyperref, geometry,mathtools} 
\mathtoolsset{showonlyrefs}

\usepackage[usenames, dvipsnames]{color}
\newtheorem{lemma}{Lemma}[section]
\newtheorem{theorem}[lemma]{Theorem}
\newtheorem{proposition}[lemma]{Proposition}
\newtheorem{corollary}[lemma]{Corollary}
\newtheorem{definition}[lemma]{Definition}
\theoremstyle{remark}
\newtheorem{remark}[lemma]{Remark}
\newtheorem{example}[lemma]{Example}

 \usepackage{mathrsfs}

\def\beq{\begin{equation}}   \def\eeq{\end{equation}}
\def\bea{\begin{eqnarray}}  \def\eea{\end{eqnarray}}
\newcommand{\lan}{\langle}
\newcommand{\ran}{\rangle}
\newcommand{\pa}{\partial}
\newcommand{\hos}{{\rm har}}
\newcommand{\per}{{\rm per}}

\renewcommand{\bar}{\overline}

\renewcommand{\uno}{{\rm Id}}
\renewcommand{\Re}{{\rm Re }}
\renewcommand{\Im}{{\rm Im }}

\renewcommand{\fA}{\mathsf{A}}
\renewcommand{\fB}{\mathsf{B}}
\renewcommand{\fC}{\mathsf{C}}
\renewcommand{\fH}{\mathsf{H}}
\renewcommand{\fK}{\mathsf{K}}
\renewcommand{\fP}{\mathsf{P}}
\renewcommand{\fQ}{\mathsf{Q}}

\renewcommand{\fV}{\mathsf{V}}

\newcommand{\Op}[1]{{\rm Op}\left(#1\right)}

\title{Growth of Sobolev norms in linear Schr\"odinger equations \\ as a dispersive phenomenon}

\author{
A. Maspero\footnote{ International School for Advanced Studies (SISSA), Via Bonomea 265, 34136, Trieste, Italy \newline
 \textit{Email: } \texttt{alberto.maspero@sissa.it}}
}

\numberwithin{equation}{section}

\begin{document}
\maketitle 

\begin{abstract}
In this paper we consider linear, time dependent Schr\"odinger equations of the form $\im \pa_t \psi = K_0 \psi + V(t) \psi$, where $K_0$ is a  strictly positive selfadjoint operator with discrete spectrum and constant spectral gaps, and $V(t)$  a time periodic potential.
We give sufficient conditions on  $V(t)$ ensuring that   $K_0+V (t)$ generates unbounded orbits.
The main condition is that the resonant average of $V(t)$, namely the average with respect to the flow of $K_0$, has a nonempty absolutely continuous spectrum and fulfills a Mourre estimate.
These conditions are  stable under perturbations. 
The proof combines pseudodifferential normal form with dispersive estimates in the form of local energy decay.\\
 We apply our abstract construction to  the Harmonic oscillator on $\R$ and to  the half-wave equation on $\T$; in each case, we provide large classes of potentials which are transporters.
\end{abstract}

\section{Introduction}


We consider  the abstract linear Schr\"odinger equation 
\begin{equation}
\label{eq1}
\im \pa_t \psi = K_0 \psi + V(t) \psi \ 
\end{equation}
on a scale of Hilbert spaces $\cH^r$; here  $V(t)$ is a time $2\pi$-periodic potential  and $K_0$  a selfadjoint, strictly positive operator with compact resolvent,  pure point spectrum and constant spectral gaps.
We prove some abstract results ensuring, 
 $\forall r >0$, the existence of   solutions $\psi(t)$ 
whose $\cH^r$-norms 
 grow polynomially fast,
$$
\norm{\psi(t)}_r \geq C_r \, \la t \ra^r , \quad \forall t \gg 1  \  ,
$$
 whereas their $\cH^0$-norms are 
constant  for all times, $\norm{\psi(t)}_0 = \norm{\psi(0)}_0$ $\forall t$.
Here   $\la t \ra:= \sqrt{1+t^2}$.\\
 These solutions therefore  exhibit
weak turbulent behavior in the form of 
energy cascade towards high frequencies. \\
We apply our abstract result to two models:  the Harmonic oscillator on $\R$ and  the half-wave equation on $\T$. In both cases we exhibit large classes of potentials $V(t)$, bounded, smooth and periodic in time, so  that the Hamiltonian $K_0 + V(t)$ generates unbounded orbits.

The phenomenon is purely perturbative: for $V = 0$ each norm of each  solution is constant for all times. 
So the central question is the existence of  potentials  able  to  transport energy to high-frequencies; we formalize this notion in the following definition:
\begin{definition}
\label{def:transporter}
 We shall say that  $V(t)$ is  a {\em transporter} if $\forall  r >0$ there exists  a  solution $\psi(t) \in \cH^r$ of \eqref{eq1} with  unbounded growth of  norm, i.e. 
\begin{equation}
\label{growth}
\limsup_{t \to \infty} \norm{\psi(t)}_r = \infty . 
\end{equation}
If this happens for  {\rm every} nonzero   solution,  we shall say that  $V(t)$ is a {\em universal transporter}.
\end{definition}
Starting with the pioneering work of Bourgain \cite{bou99}, in the last  few years there has been some efforts to construct both transporters \cite{del, FaouRaph, Thomann} 
and universal transporters \cite{BGMR1, Mas19}  for different types of  Schr\"odinger equations.
All these papers provide explicit examples of potentials, constructed ad hoc for the problem at hand.

The novelty of our  result is that we identify 
 sufficient, explicit and robust (i.e. stable under perturbations)  conditions ensuring  $V(t)$ to be  a transporter. Precisely, 
its {\em resonant average}
  \begin{align}
\label{res.av}
&\la V \ra :=  \frac{1}{2\pi} \int_0^{2\pi} e^{\im s K_0} \, V(s) \, e^{- \im s K_0} \, \di s \ 
\end{align}
must have nontrivial absolutely continuous spectrum in an interval, and over this interval it has to fulfill  a Mourre estimate   -- see \eqref{mourre.ab} below (actually we also require that both $K_0$ and $V(t)$ belong to some abstract graded algebra of pseudodifferential operators, as in \cite{BGMR2}).

The crucial point is that these conditions 
 imply dispersive estimates for $\la V \ra$ of the form
\begin{equation}
\label{dispersive}
\norm{K_0^{-k} e^{-\im t \la V \ra } P_c \phi}_0 \lesssim \la t \ra^{-k} \norm{K_0^k \phi}_0 \ , \quad \forall t \in \R , 
\end{equation}
 where $P_c$ is a    projection  on a subset of the absolutely continuous spectral space of $\la V\ra$. 
 A consequence of \eqref{dispersive} is  that we obtain solutions of $\im \pa_t \phi = \la V \ra \phi$   with decaying negative Sobolev norms and so,  by duality,   growing positive Sobolev norms.\\
The fact that Mourre estimates imply dispersive estimates as above  has origin  from the work of Sigal-Soffer in quantum scattering theory  \cite{SS} and it has been extended by many authors (see e.g.  \cite{Ski, GerSig, JenMouPer,HunSigSof, GNRS, Arbunich}). 
See also the recent results \cite{CdVS-R, CdV20, DyatlovZworski}
where similar dispersive properties are studied for pseudodifferential operators of order 0 on compact manifolds of dimension greater equal $2$.

To explain the connection between the dynamics of \eqref{eq1} and the dispersive properties of the flow of $\la V \ra$, let us briefly  describe the main ideas of the proof. 
The first step is to put system \eqref{eq1} into its resonant pseudodifferential normal form. 
This is the resonant variant of the normal form developed in \cite{BGMR2} for non-resonant systems (and essentially an abstract version of the normal form of Delort \cite{del}); it allows, $\forall N \in \N$,  to conjugate equation \eqref{eq1} to 
\begin{equation}
\label{eq.2}
\im \pa_t \vf = \big(K_0 + Z^{(N)}(t) + R^{(N)}(t)  \big) \vf
\end{equation}
where $Z^{(N)}(t)$ is a time dependent operator fulfilling 
\begin{equation}
\label{Z.prop}
\im \pa_t Z^{(N)}(t) = [K_0, Z^{(N)}(t)] , \qquad Z^{(N)}(0) = \la V \ra + \mbox{lower order terms}
\end{equation}
whereas $R^{(N)}(t)$ is an $N$-smoothing operator (it maps $\cH^{r} \to \cH^{r+N}$ continuously  $\forall r$). 
The difference with the non-resonant case of  \cite{BGMR2} is that, in that paper,  $Z^{(N)}(t)$ commutes with $K_0$. 
This is not anymore true in the resonant case we deal with;
 however \eqref{Z.prop} implies that 
$e^{\im t K_0} Z^{(N)}(t) e^{- \im t K_0}$ is time independent and thus  coincides with $ Z^{(N)}(0)$. 
Thus, conjugating \eqref{eq.2} with $e^{- \im t K_0}$, we arrive at the equation 
\begin{equation}
\label{eq.3}
\im \pa_t \phi = \big(\la V \ra + T + R(t) \big) \phi
\end{equation}
where $T:= Z^{(N)}(0)- \la V \ra$ is a time independent selfadjoint compact operator   and $R(t)$ is  $N$-smoothing. 

Then we  analyze the dynamics of the truncated equation
\begin{equation}
\label{intro.tr}
\im \pa_t \phi = \big( \la V \ra + T  \big) \phi
\end{equation}
and prove that it has  solutions with decaying negative  Sobolev norms and so, by duality, growing positive Sobolev norms.
 This is the core of the proof; after this step, it is not difficult to construct a solution of the complete equation \eqref{eq.3} exhibiting energy cascade, exploiting that $R(t)$ is  regularizing.
So let us concentrate on \eqref{intro.tr}. 
The goal is to prove a dispersive estimate of the form \eqref{dispersive} with 
 $\la V\ra$ replaced by  $\la V \ra + T$. 
However 
the point is delicate because the absolutely continuous spectrum of $\la V\ra$ (which exists by assumption) could be completely destroyed by adding  $T$; a celebrated theorem by Weyl-von Neumann ensures that  {\em any} selfadjoint operator  (in a separable Hilbert space) can be perturbed by a (arbitrary small) compact  selfadjoint operator so that its spectrum becomes  pure point  (see e.g. \cite[pag. 525]{Kato}). 
This is exactly the situation we want to avoid, as pure point spectrum prevents 
dispersive estimates. 
To get around this,   we exploit that Mourre estimates are stable under  pseudodifferential perturbations.
This  allows us to prove that  $\la V \ra + T$ fulfills Mourre estimates  and thus
a dispersive estimate as \eqref{dispersive}.

We also stress  that fulfilling a Mourre estimate seems to be  a  quite general  condition, and  in the applications we  exhibit large classes of operators which are transporters. 
 For example, for the  half wave equation we  prove that any operator of the form
 $\cos(mt) v(x)$ with $v \in C^\infty(\T, \R)$ and  $m \in \Z$ is a transporter provided the $m$-th Fourier coefficient of $v(x)$ is not zero.

Finally,  the conditions we identity 
to be transporters are robust: if a potential $V(t)$ fulfills them, so does $V(t) + W(t)$ for {\em any} sufficiently small pseudodifferential operator $W(t)$. 
This shows that  weakly turbulent phenomena induced by certain transporters  are stable under perturbations.
Up to our knowledge, this ``stability of instability'' is new in the literature and we consider it one of the main novelty of the paper.\\

We conclude the introduction by reviewing the known results about existence of transporters for  linear time dependent Schr\"odinger equations. 
As we already mentioned, the first result is due to Bourgain \cite{bou99}, who constructed a transporter for the Schr\"odinger equation on the torus; in this case $V(t,x)$ is a bounded real analytic function.
Delort \cite{del} constructs a transporter for the harmonic oscillator on $\R$, which is a time $2\pi$-periodic pseudodifferential operator of order zero.
In \cite{BGMR1} we proved that  $a x \sin( t)$, $a >0$,  is a universal transporter for the harmonic oscillator on $\R$; in this case the potential is an unbounded operator.
In \cite{Mas19} we constructed universal transporters for the abstract equation \eqref{eq1}, and applied the result to the harmonic oscillator on $\R$, the half-wave equation on $\T$ and on a Zoll manifold; in all cases the universal transporters are time periodic pseudodifferential operators of order 0.
Finally recently Faou-Raphael \cite{FaouRaph} 
 constructed a transporter for the harmonic oscillator on $	\R$ which is a time dependent function (and not a pseudodifferential operator), and Thomann \cite{Thomann} has constructed a transporter for the harmonic oscillator on  the Bargman-Fock space.
 Finally we recall the long-time growth result \cite{HausMaspero} for the  semiclassical 
anharmonic oscillator on $\R^d$.

\vspace{1em}

\noindent{\bf Acknowledgments:} We  thank Matteo Gallone for
helpful discussions on spectral theory and  Dario Bambusi and  Didier Robert for useful suggestions during the preparation of this work.

\section{The abstract result}
\label{sec:alg}
We start with a Hilbert space $\cH$, endowed  with the scalar product $\la \cdot, \cdot \ra$,   and a reference operator $K_0$,
which we assume to be  selfadjoint,   positive, namely such that
$$
\langle \psi; K_0\psi\rangle\geq c_{K_0} \norm{\psi}^2\ ,\quad \forall
\psi\in D(K_0^{1/2})\ ,\quad c_{K_0}>0\ , 
$$ and 
with compact resolvent. \\
We define as usual a scale of Hilbert spaces by $\cH^r:=D(K_0^r)$
(the domain of the operator $K_0^r$) if $r\geq 0$, and
$\cH^{r}=(\cH^{-r})^\prime$ (the dual space) if $ r<0$.  Finally we
denote by $\cH^{-\infty} = \bigcup_{r\in\R}\cH^r$ and $\cH^{+\infty} =
\bigcap_{r\in\R}\cH^r$.  We endow $\cH^r$ with the natural norm
$\norm{\psi}_r:= \norm{(K_0)^r \psi}_{0}$, where $\norm{\cdot}_0$ is
the norm of $\cH^0 \equiv \cH$. Notice that for any $m\in\R$,
$\cH^{+\infty}$ is a dense linear subspace of $\cH^m$ (this is a
consequence of the spectral decomposition of $K_0$).

\begin{remark}
\label{rem:K0.flow}
By the very definition of $\cH^r$, the unperturbed flow $e^{-\im t K_0}$ preserves each norm, $\norm{e^{-\im t K_0} \psi }_r = \norm{\psi}_r$ $\, \forall t \in \R$. 
Consequently, every orbit of  equation \eqref{eq1} with $V(t) = 0$  is bounded.
\end{remark}

Following \cite{BGMR2},  we introduce now a graded algebra $\cA$ of operators which mimic some
fundamental properties of different classes of pseudodifferential
operators.  For $m\in\R$ let $\cA_m$ be a linear subspace of $
\bigcap_{s\in\R}\cL(\cH^s,\cH^{s-m})$ and define
$\cA:=\bigcup_{m\in\R}\cA_m$.  We notice that the space
$\bigcap_{s\in\R}\cL(\cH^s,\cH^{s-m})$ is a Fr\'echet space equipped
with the semi-norms: $\Vert A\Vert_{m,s} := \Vert
A\Vert_{\cL(\cH^s,\cH^{s-m})}$.

We shall need  to control the smoothing properties of the operators
in the scale $\{\cH^r\}_{r\in\R}$.  If $A\in\cA_m$ then $A$ is more
and more smoothing if $m\rightarrow -\infty$ and the opposite as
$m\rightarrow +\infty$. We will say that $A$ is of {\em order $m$} if
$A \in \cA_m$.
    \begin{definition}
    \label{smoothing}
     We say  that $S\in\cL(\cH^{+\infty},\cH^{-\infty} )$ is $N$-smoothing if  $\forall \kappa \in \R$, it  can be extended  to an operator in   $\cL(\cH^{\kappa}, \cH^{\kappa+N})$.  When this is true  for  every $N\geq 0$, we say that $S$  is a smoothing  operator.
    \end{definition}
    The first set of assumptions concerns the properties of  $\cA_m$:\\
    
     \noindent
     {\bf Assumption I: Pseudodifferential algebra}  
     \begin{itemize}
      \item[(i)] For each $m\in \R$, $K_0^m\in\cA_m$; in particular
        $K_0$ is an operator of order one.
\item[(ii)]  For each $m\in\R$, $\cA_m$ is a Fr\'echet space for a family of  filtering semi-norms $\{ \wp^m_j \}_{j\geq 0}$  such that the embedding $\cA_m\hookrightarrow  \bigcap_{s \in \R} \cL(\cH^s,\cH^{s-m})$ is continuous\footnote{A family of seminorms $\{ \wp^m_j \}_{j\geq 0}$  is called filtering if for any $j_1, j_2 \geq 0$  there exist  $k \geq 0$ and $c_1, c_2 >0$  such that the two inequalities $\wp^m_{j_1}(A) \leq c_1 \wp^m_{k}(A)$ and $\wp^m_{j_2}(A) \leq c_2 \wp^m_{k}(A)$ hold for any $A \in \cA_m$.}.\\
 If $m^\prime \leq m$ then $\cA_{m^\prime}\subseteq\cA_m$ with a continuous embedding.

    \item[(iii)]    $\cA$ is a graded algebra,  i.e. $\forall m,n\in \R$:  if $A\in \cA_m$   and $B\in\cA_n$ then $A B\in\cA_{m+n}$ 
     and the map $(A,B)\mapsto AB$ is continuous from 
     $\cA_{m}\times\cA_{n}$ into $\cA_{m+n}$.
     \item[(iv)] $\cA$ is a graded Lie-algebra\footnote{This property will impose the choice of the semi-norms  $\{ \wp^m_j \}_{j\geq 1}$. We will see in the examples that the natural  choice $(\Vert \cdot\Vert_{m,s})_{s\geq 0}$ has to be refined.  } : if $A\in \cA_m$   and $B\in\cA_n$ then the commutator $[A,B]\in\cA_{m+n-1}$  and the map $(A,B)\mapsto [A,B]$
      is continuous from $\cA_{m}\times\cA_{n}$ into $\cA_{m+n-1}$.
      
     \item[(v)] $\cA$ is closed under  perturbation by smoothing operators in the following sense:
     let $A$ be a linear map: $\cH^{+\infty}\rightarrow\cH^{-\infty}$. If   there  exists $m\in\R$ such that  for  every $N>0$  we  have  a   decomposition        $A= A^{(N)}+ S^{(N)}$,  with $A^{(N)}  \in\cA_m$  and $S^{(N)}$  is  $N$-smoothing,  then
      $A\in\cA_m$. 
       \item[(vi)] If $A \in \cA_m$ then also the adjoint operator $A^* \in \cA_m$.
  The  duality here is defined by the scalar product 
      $\lan\cdot, \cdot\ran$ of $\cH=\cH^0$. The adjoint $A^*$ is defined by 
      $\lan u, Av\ran = \lan A^*u, v\ran$  for $u,v\in\cH^\infty$  and extended by continuity.  
     \end{itemize}
    
    It is well known that classes of pseudodifferential operators
    satisfy these properties, provided one chooses for $K_0$ a
    suitable operator of the right order (see e.g. \cite{ho}). 
\begin{remark}
\label{rem:control}
One has that $ \forall A\in\cA_m$, $\forall B\in\cA_n$
\begin{align}
\label{est.1}
\forall m, s  \quad \exists N \ s.t.\  &\norm{A}_{m,s} \leq C_1 \,
\wp^m_N(A)  \ ,
\\
\label{est.2}
\forall m, n, j \quad \exists N \ s.t.\  &\wp^{m+n}_j(AB) \leq C_2 \, \wp^m_N(A)
\, \wp^n_N(B)  \ , 
\\
\label{est.3}
\forall m, n, j\quad  \exists N \ s.t.\   &\wp^{m+n-1}_j( [A, B] ) \leq C_3 \, \wp^m_N(A) \, \wp^n_N(B)  \ , 
\end{align}
for some positive
constants $C_1(s,m)$, $C_2(m,n,j)$, $C_3(m,n,j)$.
\end{remark}

\begin{remark}
\label{rem:compact}
Any $A \in \cA_{m}$ with $m <0$ is a compact operator on $\cH$. \\
Indeed write  $A = A K_0^{-m} \, K_0^{m}$. Then  $ A K_0^{-m} \in \cA_0$ is a bounded operator on $\cH$ (Assumption I (i)--(iii)), whereas  $K_0^{m} \equiv  (K_0^{-1})^{-m}$ is compact on $\cH$, as  $K_0^{-1}$ is a compact operator by assumption.
\end{remark}

For $\Omega\subseteq \R^d$ and $\cF$ a Fr\'echet space, 
we will denote by $C_b^m(\Omega, \cF)$ the space of
$C^m$ maps $f: \Omega\ni x\mapsto f(x)\in\cF$ such that,  for every
seminorm $\norm{\cdot}_j$ of $\cF$, one has
       \begin{equation}
       \label{star}
       \sup_{x\in\Omega}\Vert\partial_x^\alpha f(x)\Vert_{j} < +\infty
       \ , \quad \forall \alpha\in \N^d\  :\ \left|\alpha\right|\leq m  \ .
       \end{equation}
If \eqref{star} is true
$\forall m$, we say $f \in C^\infty_b(\Omega, \cF)$. 
Similarly we denote by $C^\infty(\T, \cF)$ the space of smooth maps from the torus $\T = \R/(2\pi\Z)$ to the Fr\'echet space $\cF$. \\

The second set of assumptions  concerns the operator $K_0$,  its spectral structure and an  Egorov-like property, also well
known for pseudo-differential operators.  \\

     \noindent
       {\bf Assumption II: Properties of $K_0$}
       \begin{itemize}
       \item[(i)] The operator $K_0$ has purely discrete spectrum fulfilling       
         \begin{equation}
 \label{AssA}
 {\rm spec}(K_0) \subseteq \N +\lambda
 \end{equation}   for some  $\lambda \geq 0$.
 \item[(ii)] For any $m \in \R$ and $A\in\cA_m$,   the map defined on $\R$ by 
         $\tau\mapsto A(\tau):={\rm e}^{\im \tau K_0}\, A \, {\rm e}^{-\im \tau
         K_0}$ belongs to $ C^\infty_b(\R, \cA_m)$ and one has
  \begin{equation}
  \label{est.4}
  \forall  j \quad \exists N \ s.t.\  \ \sup_{\tau \in \R}  \wp^{m}_j(A(\tau)) \leq C_4 \, \wp^m_N(A) 
  \end{equation}
  for some positive constant $C_4(m,j)$.
\end{itemize}
\vspace{1.5em}

\begin{remark}
\label{K0.flow.p}
Assumption {\rm II} (i)  guarantees that
$e^{\im 2\pi K_0} = e^{\im 2 \pi \lambda}$.
As a consequence,  for any operator $V$, the map  $\tau \mapsto e^{\im \tau K_0} V e^{-\im \tau K_0}$ is $2\pi$-periodic.
\end{remark}


\vspace{1em}
We denote by $C^\infty_c(\R^d, \R_{\geq 0})$ the set of smooth  functions with compact support from $\R^d$ to $\R_{\geq 0}$ (hence non-negative). 
Furthermore from now on, given  two operators $\fA, \fB \in \cL(\cH)$, we write $\fA \leq \fB$  with the meaning   $\la \fA \vf, \vf \ra \leq \la \fB \vf, \vf \ra$ $\, \forall \vf \in \cH$.\\

The last  set of assumptions concerns  the resonant average $\la V\ra$   of the potential $V(t)$ (see  \eqref{res.av}) and its spectrum $\sigma(\la V \ra)$.
 Note that if $V(t)$ is selfadjoint $\forall t$, so is $\la V \ra$. \\
 
    \noindent
       {\bf Assumption III: Properties of the potential $V(t)$} \\ 
      The operator   $V \in C^\infty(\T, \cA_{0})$, $V(t)$  selfadjoint $\forall t$, and  its resonant average $ \la V \ra$ fulfills:
      \begin{itemize}
      \item[(i)]  There exists  an interval  $I_0\subset \R$  such that  $\abs{\sigma(\la V \ra)\cap I_0} >0 $; here $\abs{\cdot}$ denotes the Lebesgue measure.
      \item[(ii)]  {\em Mourre estimate} over  $I_0$:  there exist a  selfadjoint  operator $A \in \cA_{1}$ and a function   $g_{I_0} \in C_c^\infty(\R, \R_{\geq 0})$ with $g_{I_0}\equiv 1$ on $I_0$ such that 
\begin{equation}
\label{mourre.ab}
g_{I_0}(\la V \ra) \, \im [\la V \ra, A] \, g_{I_0}(\la V \ra ) \geq \theta \, g_{I_0}(\la V \ra)^2  + K
\end{equation}
for some $\theta >0$ and $K$ a selfadjoint compact operator.
      \end{itemize}
The operator $g_{I_0}(\la V \ra)$ above is   defined via functional calculus, see Appendix \ref{app:fc}. \\
Following the literature, we  shall say that $\la V \ra$ is {\em conjugated to} $A$ {\em over} $I_0$.
 
 \begin{remark}
By Mourre theory \cite{Mourre}
 $\la V \ra$ has,  in the interval $I_0$, a  nontrivial absolutely continuous spectrum  with finitely many eigenvalues of finite multiplicity and no singular continuous spectrum. In  general one cannot exclude the existence of  embedded eigenvalues in the absolutely continuous  spectrum.\footnote{ For 
 example consider $H \in \cL(L^2(\T))$ given by
 $$
 (H u)(x) :=  \cos(x)  u(x)  + \delta (1 - \delta^{-1} \cos(x))\frac{1}{2\pi} \int_\T  u(x) \big( 1 -\delta^{-1}  \cos(x)\big) \, \di x \ , \quad
 \delta \in \left(-\frac12, \frac12\right)\setminus \{0\} \ . 
 $$
$H$ is selfadjoint,  a 1-rank perturbation of the multiplication operator by $\cos(x)$, it has absolutely continuous spectrum in the interval $(-1,1)$,   and $\delta$ is an embedded  eigenvalue with eigenvector 
$u(x) \equiv 1$.
Moreover $H$  is  conjugated to $\sin(x)\frac{\pa_x}{\im} + \frac{\pa_x}{\im} \sin(x) $ over $[-\frac12, \frac12]$.}
 \end{remark}

We are ready to state our main results. The first  says that, under the set of assumptions above,  $V(t)$ is a transporter  in the sense of Definition \ref{def:transporter}:
\begin{theorem}
\label{thm:ab0}
Assume that $\cA$ is a graded algebra as in {\rm Assumption {\rm I}}, and that $K_0 $ and $V(t) \in C^\infty(\T, \cA_0)$ satisfy {\rm Assumptions} {\rm II} and {\rm III}.  Then $V(t)$ is a transporter for the equation 
\begin{equation}
\label{eq.ab0}
\im \pa_t \psi = (K_0 + V(t)) \psi  \ .
\end{equation}
More precisely, 
for any  $r >0$  there exist a solution $\psi(t)$ of \eqref{eq.ab0} in $\cH^r$ and constants $C , T  >0$ such that 
\begin{equation}
\label{eq:gr0}
\norm{\psi(t)}_{r} \geq C \la t \ra^{r} , \quad \forall t \geq  T \ .
\end{equation}
\end{theorem}

We also  prove  a  stronger result: namely not only $V(t)$ is a transporter, 
but  also   {\em any}  operator sufficiently close  to it  (in the $\cA_0$-topology). 
Here the precise statement:


\begin{theorem}
\label{thm:ab}
With the same assumptions of Theorem \ref{thm:ab0}, 
there exist  $\epsilon_0 > 0$ and  $\tM \in \N$ such that  for any $W \in C^\infty(\T, \cA_{0})$, $W(t)$ selfadjoint $\forall t$,   fulfilling  
\begin{equation}
\label{normW}
\sup_{t\in \T} \, \wp^0_\tM (W(t)) \leq \epsilon_0 , 
\end{equation}
then $V(t) + W(t)$ is a transporter for the equation
\begin{equation}
\label{eq.ab}
\im \pa_t \psi = \big( K_0 + V(t) +  W(t) \big)\psi  \ .
\end{equation}
More precisely, for any  
 $r >0$  there exist a solution $\psi(t)$ in $\cH^r$ of \eqref{eq.ab}
  and constants $C, T >0$ such that 
\begin{equation}
\label{eq:gr}
\norm{\psi(t)}_{r} \geq C \la t \ra^{r} , \quad \forall t \geq  T \ .
\end{equation}
\end{theorem} 
 
 Let us  comment the above  results.
 \begin{enumerate}
 \item The growth of Sobolev norms of Theorem \ref{thm:ab0}  is truly an energy cascade  phenomenon;  indeed the $\cH^0$-norm of any solution of 
 \eqref{eq.ab0}  is preserved for all times, $\norm{\psi(t)}_0  = \norm{\psi(0)}_0$, $\, \forall t \in \R$. This is  due to the selfadjointness of $K_0 + V(t)$ (the same happens to solutions of \eqref{eq.ab}).

\item
Estimates \eqref{eq:gr0}, \eqref{eq:gr} provide optimal lower bounds for the speed of growth of the Sobolev norms. 
Indeed we proved  \cite{MaRo} that, under the assumptions above\footnote{ in particular the fact that $[K_0, V(t)]$ and $[K_0, V(t) + W(t)]$ are  uniformly (in $t$) bounded operators on the scale $\cH^r$}, {\em any} solution of \eqref{eq.ab0} or \eqref{eq.ab} fulfills the upper bounds
$$
\forall r>0 \ \ \ \exists \,\wt  C_r > 0 \colon 
\quad 
\norm{\psi(t)}_r \leq \wt C_r \la t \ra^r \, \norm{\psi(0)}_r .
$$
Thus, Theorems \ref{thm:ab0}, \ref{thm:ab} construct unbounded solutions with  optimal growth. 

\item Theorem \ref{thm:ab} proves robustness of certain type of transporters under small pseudodifferential perturbations.
This shows a sort of ``stability of instability'', which, up to our knowledge, is new in this context.

\item Actually there are  infinitely many distinct solutions undergoing  growth of Sobolev norms.
Their initial data are constructed in a unique way starting  from functions  belonging to  the absolutely continuous spectral subspace of the operator $\la V\ra$.
We describe such initial data in  Corollary \ref{cor:inf.many}.

\item Energy cascade is  a resonant phenomenon; here it happens because $V(t)$ oscillates at frequency $\omega = 1$  which resonates with the spectral gaps of $K_0$.  
In \cite{BGMR2} we proved  that if $V(t) \equiv \fV(\omega t)$ is quasiperiodic in time with a frequency vector $\omega \in \R^n$ fulfilling the non-resonant condition
$$
\exists \gamma, \tau >0 \colon \quad \abs{\ell + \omega \cdot k} \geq \frac{\gamma}{\la k \ra^\tau} \quad
\forall \ell, k \in \Z \times \Z^n \setminus \{0 \} 
$$
(which is violated if $V(t)$ is $2\pi$-periodic) then the Sobolev norms grow at most as $\la t \ra^\epsilon$ $\forall \epsilon >0$. 
The $\la t \ra^\epsilon$-speed of growth is also known  for systems with increasing  \cite{nen, MaRo, BGMR2} or shrinking \cite{duclos, MONTALTO2019} spectral gaps  and  for Schr\"odinger equation on $\T^d$ with bounded \cite{bourgain99,del2, BERTI2019} and even unbounded  \cite{bambusi2020growth} potentials.

\item In  concrete models  one can typically   prove that if  $V(t)$  is sufficiently small in size and oscillates in time with  a  strongly non resonant frequency $\omega$ (typically belonging to some Cantor set of large measure), then all  solutions have uniformly in time bounded Sobolev norms. 
Therefore the stability/instability of the system depends only on the resonance property of the frequency $\omega$.
 We mention just the recent results \cite{Bam16I,BGMR1} which deal with the harmonic oscillator (as we  consider it in the applications) and refer to those papers for a complete bibliography.

\item The most delicate assumption  to verify is  \eqref{mourre.ab}. In the applications, one can try to construct  an escape function for  the principal symbol
$\la v \ra$ of  $\la V \ra$.
This means to find a  symbol $a(x,\xi)$ of order 1 such that  the Poisson bracket $\{ \la v \ra, a \}$ is strictly positive in some energy levels:
$$
\exists c >0 \colon \quad \{ \la v \ra, a \} \geq c \qquad \mbox{ in } \{(x, \xi): \ \    \abs{\la v \ra(x,\xi) - \lambda } \leq \delta \} \ . 
$$
Then  symbolic calculus and sharp G{\aa}rding inequality imply that  \eqref{mourre.ab} holds in the interval $I= (\lambda - \delta, \lambda+\delta)$; see \cite{CdVS-R} Section 6.2 for details.

\end{enumerate}

We finally note  that the second theorem is stronger than the first one
and  implies  it  in the special case $W(t) \equiv 0$. 
However we think that the statement of Theorem \ref{thm:ab0} is clear and useful in the applications (see e.g.  Section \ref{sec:app}), so we decided to state it on its own. 
Having said so,  in the sequel we shall only prove Theorem \ref{thm:ab}.
 
\section{Proof of the abstract result}
 As already mentioned, we shall only prove Theorem \ref{thm:ab}.
The proof is divided in three steps; in the first one we put system  \eqref{eq.ab} in its resonant pseudodifferential normal form.
In the second one we analyze the dynamics of the effective Hamiltonian and prove  the existence of solutions with decaying negative Sobolev norms. The final step is to construct a solution of the complete equation exhibiting growth of Sobolev norms.

\subsection{Resonant pseudodifferential normal form} 
The goal of this section is to put  system \eqref{eq.ab} into its resonant pseudodifferential normal form up to an arbitrary $N$-smoothing operator. 
In this first step we shall only require Assumptions I and II.
It is slightly more convenient to deal with the equation
\begin{equation}
\label{eq:v}
\im \pa_t \psi = \big( K_0 + \fV(t) \big) \psi
\end{equation}
and then to specify the result for 
 $\fV(t) = V(t) + W(t)$ as in  \eqref{eq.ab}. 
 	Given  $\fV \in C^\infty(\T, \cA_m)$,  $m \in \R$,  we define   the {\em averaged operator}
\begin{align}
\label{av.op}
& \wh  \fV (t) := \frac{1}{2\pi} \int_0^{2\pi} e^{\im s K_0} \, \fV(t+s) \, e^{- \im s K_0} \, \di s \ . 
\end{align}
We shall prove below that $\wh  \fV (t) \in C^\infty(\T, \cA_m)$ (see Lemma \ref{prop:av}).
\begin{proposition}[Resonant pseudodifferential  normal form]
\label{lem:rpnf2}
Consider equation \eqref{eq:v} with $\fV \in C^\infty(\T, \cA_0)$,  $\fV(t)$ selfadjoint $\forall t$.
There exists a sequence 	$\{ X_j(t)\}_{j \geq 1}$ of selfadjoint (time-dependent) operators in $\cH$ with $X_j \in C^\infty(\T , \cA_{1-j})$ and fulfilling 
\begin{equation}
\label{unitary2}
\forall r \in \R , \ \exists c_{r,j}, C_{r,j} >0 \colon \qquad c_{r,j} \norm{\vf}_r  \leq \norm{e^{\pm \im X_j(t)}\vf}_r \leq C_{r,j} \norm{\vf}_r , \qquad \forall t \in \R ,
\end{equation} 
such that the following holds true. For any $N \geq 1$, the change of variables 
\begin{equation}
\label{change.v1}
\psi = e^{- \im X_1(t)} \cdots e^{- \im X_N(t)}  \vf  
\end{equation}
transforms \eqref{eq:v} into the equation
\begin{equation}
\label{eq.res}
\im \pa_t \vf = \big(K_0 + Z^{(N)}(t) + \fV^{(N)}(t)\big) \vf \ ; 
\end{equation}
here  $\fV^{(N)} \in C^\infty(\T, \cA_{-N})$ whereas $Z^{(N)} \in C^\infty(\T, \cA_0)$,   it is  selfadjoint $\forall t$,  it   fulfills 
\begin{equation}
\label{Z.prop2}
\im \pa_t Z^{(N)}(t) = [K_0, Z^{(N)}(t)]
\end{equation}
and it has the expansion
\begin{equation}
\label{Z.exp}
Z^{(N)}(t) = \wh \fV(t)  +  T^{(N)}(t) ,  \qquad T^{(N)} \in  C^\infty(\T, \cA_{-1}) \ .
\end{equation}
Here  $\wh \fV(t)$ is the averaged operator defined in \eqref{av.op}. 
\end{proposition}

In order to prove the proposition we start with some preliminary results. The first regards the  properties of $\wh \fV(t)$.

\begin{lemma}
\label{prop:av}
Let $\fV \in C^\infty(\T, \cA_m)$, $m \in \R$,  $\fV(t)$ selfadjoint $\forall t$. Then the following holds true.
\begin{itemize}
\item[(i)] $\wh \fV  \in C^\infty(\T, \cA_m)$, it is selfadjoint $\forall t$,  it commutes with $\im \pa_t - K_0$, i.e. $ \im \pa_t  \wh  \fV(t) = [ K_0 , \wh  \fV(t)] $ and 
\begin{equation}
\label{B-est}
\forall j, \ell \geq 0   \quad \exists \, M\in \N  ,  \ C >0  \ \ {\rm s.t.} \ \ \sup_{t \in \T}  \wp^m_j(\pa_t^\ell \wh \fV(t)) \leq C \,  \sup_{t \in \T } \  \wp^m_M(\fV(t) )  \ . 
\end{equation}
\item[(ii)] The resonant averaged operator $\la \fV \ra$, defined in \eqref{res.av}, belongs to $\cA_m$, it is selfadjoint and 
\begin{equation}
\label{B-est2}
\forall j  \geq 0 \quad \exists \, M \in \N, \  C >0  \ \ {\rm s.t.} \ \  \wp^m_j(\la \fV \ra) \leq C  \,  \sup_{t \in \T }   \  \wp^m_M( \fV(t) ) \ . 
\end{equation}
\item[(iii)] One has the chain of identities
\begin{equation}
\label{av.W}
\wh \fV(0)  = \la \fV \ra  = e^{\im t K_0} \,\wh \fV(t) \, e^{- \im t K_0}  =  \langle \,  \wh \fV \, \rangle , \qquad \forall t \in \R \ . 
\end{equation}
\end{itemize}
\end{lemma}
\begin{proof}
$(i)$ The properties  $\wh \fV  \in  C^\infty(\T, \cA_m)$ and $\wh \fV(t)$ selfadjoint $\forall t$  follow from  Assumption II   and the fact that   $ \fV (t)$ is $2\pi$-periodic in $t$ and selfadjoint $\forall t$.
Let us prove it commutes with  $\im \pa_t - K_0$. Using 
$$\pa_s \left( e^{\im s K_0 } \, \fV (t+s) \, e^{- \im s K_0}\right)  = e^{\im s K_0} \big( \im [K_0 , \fV (t+s)] + \pa_s \fV (t+s) \big) e^{- \im s K_0}$$  we get
\begin{align*}
 \pa_t \wh \fV  (t)
 & =  \frac{1}{2\pi} \int_0^{2\pi} e^{\im s K_0} \,  \pa_t \fV (t+s) \, e^{- \im s K_0} \, \di s  = 
\frac{1}{2\pi} \int_0^{2\pi} e^{\im s K_0} \,  \pa_s  \fV (t+s) \, e^{- \im s K_0} \, \di s \\
& =  \frac{1}{2\pi \im }\int_0^{2\pi} e^{\im s K_0} \,  [K_0,   \fV (t+s)] \, e^{- \im s K_0} \, \di s 
= \im^{-1} \, [ K_0,  \wh \fV(t)   ]
\end{align*}
where in the second line we used  the periodicity of $s \mapsto e^{\im s K_0 } \, \fV (t+s) \, e^{- \im s K_0}$ (see Remark \ref{K0.flow.p}) to remove the boundary terms. 
Estimate \eqref{B-est} for $\ell = 0$  follows from Assumption II.
  For $\ell \geq 1$ we use induction: assume   \eqref{B-est} is true  up to  a certain  $\ell$; using  $\pa_t^{\ell+1}\wh \fV(t) = - \im \pa_t^\ell [K_0,\wh \fV(t)] = - \im [K_0, \pa_t^\ell \wh\fV(t)]$, we get $\forall j \in \N$
  $$
  \wp_j^m(\pa_t^{\ell+1}\wh\fV(t)) \leq \ \wp_{j}^m([K_0, \pa_t^\ell \wh\fV(t)]) \leq C  \wp_{j_1}^m( \pa_t^\ell \wh\fV(t)) \leq  C  \wp_{j_2}^m( \fV(t))
  $$
  using also the inductive assumption. This proves \eqref{B-est}.
\\
$(ii)$ It is clear that $\la \fV  \ra $ is time independent, selfadjoint and in  $\cA_m$ by Assumption II.  Estimate \eqref{B-est2} follows from Assumption II.\\
$(iii)$  Clearly  $\wh \fV(0) = \la \fV \ra$.
Then, as the map $\tau \mapsto e^{\im \tau K_0} \, \fV (\tau) \, e^{- \im \tau K_0}$ is $2\pi$-periodic, one has $\forall t \in \R$
$$
 e^{\im t K_0}\,  \wh \fV(t) \, e^{- \im t K_0} = \frac{1}{2\pi} \int_0^{2\pi} e^{\im (t+s) K_0} \, \fV(t+s) \, e^{- \im (s+t) K_0} \, \di s = \la \fV \ra \ . 
 $$
Finally, exploiting this last identity, one has 
$$
\langle \,  \wh \fV \, \rangle = \frac{1}{2\pi}\int_0^{2\pi} e^{\im t K_0} \, \wh \fV (t) \, e^{- \im t K_0} \di t  = \frac{1}{2\pi} \int_0^{2\pi} \la \fV  \ra \di t = \la \fV  \ra 
$$
which completes the proof of  \eqref{av.W}.
\end{proof}

The second preliminary result regards how to solve the homological equations which appear during the normal form procedure. 
More precisely we look for  a  time periodic operator $X(t)$ solving the homological equation 
\begin{equation}
\label{hom}
\pa_t X(t)  + \im [K_0, X(t)] = \fV(t) - \wh \fV(t) , 
\end{equation}
where $\wh \fV(t)$ is the averaged operator defined in \eqref{av.op}.
This is done in the next lemma.

\begin{lemma}
\label{lem:he}
Let $\fV \in C^\infty(\T, \cA_m)$, $m \in \R$, $\fV(t)$ selfadjoint $\forall t$. The homological equation \eqref{hom} has  a solution $X \in  C^\infty(\T, \cA_m)$ and $X(t)$ is selfadjoint $\forall t$.
\end{lemma}
\begin{proof}
We look for a solution of \eqref{hom} using the method of variation of constants. 
In particular we take $X(t)  = e^{- \im t K_0} \, Y(t) \, e^{\im t K_0}$ for some $Y \in C^\infty(\R, \cA_m)$
with $Y(0) = 0$  to be determined.
  Then $X$ solves \eqref{hom} provided $  \pa_t Y(t) = e^{\im t K_0} \, (\fV(t) - \wh \fV(t) ) \, e^{-\im t K_0} $, giving 
$$
Y(t) = \int_0^t e^{\im s K_0} \big( \fV(s) - \wh \fV (s) \big) \, e^{- \im s K_0} \, \di s .
$$
By  Lemma \ref{prop:av} and Assumption II, $Y \in C^\infty(\R, \cA_m)$ and it is selfadjoint  $\forall t$. Therefore one gets
$$
X(t) = \int_0^t e^{\im (s-t) K_0} \big( \fV(s) - \wh \fV (s) \big) \, e^{- \im (s-t) K_0} \, \di s .
$$
Again $X \in C^\infty(\R, \cA_m)$ and it is selfadjoint $\forall t$. Finally (recall Remark \ref{K0.flow.p})
\begin{align*}
X(t+2\pi) - X(t) 
& =    \int_t^{t+2\pi} e^{\im (s-t) K_0}\, \big( \fV(s) - \wh \fV(s) \big)\, e^{- \im (s-t) K_0} \, \di s \\
& = e^{- \im t K_0} \int_0^{2\pi}  e^{\im s K_0}\, \big( \fV(s) - \wh \fV(s) \big)\, e^{- \im s K_0} \, \di s  \ e^{\im t K_0} \\
 & =  2\pi e^{- \im t K_0} \big( \la \fV \ra - \langle \, \wh \fV  \, \rangle \big) e^{\im t K_0} 
 \stackrel{\eqref{av.W}}{=} 0 
\end{align*}
which proves the periodicity of $t \mapsto X(t)$.
\end{proof}

We are ready to prove Proposition \ref{lem:rpnf2}. 
During the proof we shall use some results proved in \cite{BGMR2} about the flow generated by   pseudodifferential operators;   we collect them, for the reader's convenience,  in  Appendix \ref{app:flow}.

\begin{proof}[Proof of Proposition \ref{lem:rpnf2}]
The proof is  inductive on $N$.  Let us start with $N = 1$.
We look for a change of variables of the form $\psi= e^{ -\im X_1(t)}
\vf$ where $X_1(t)\in C^\infty(\T, \cA_{0 })$ is  selfadjoint  $\forall t$,  to be determined.
By Lemma \ref{T.1},  $\vf$ fulfills the Schr\"odinger equation
$\im \pa_t \vf = H^+(t) \vf$ with
\begin{align*}
\label{4.1.1}
H^{+}(t)&:= e^{\im X_1(t)} \, \big( K_0 + \fV(t) \big)\, e^{-\im X_1(t)}
-\int_0^1 e^{\im s X_1(t) } \,(\pa_t X_1(t) ) \, e^{-\im
  s X_1(t) }  \ \di s  \ . 
  \end{align*}
  Then a commutator  expansion, see Lemma \ref{comX}, gives
  \begin{align*}
  H^{+}(t)
  &= K_0 + \im [ X_1(t), K_0]  + \fV(t)  - \pa_t  X_1 +    \fV^{(1)}(t) 
\end{align*} 
with $\fV^{(1)}  \in  C^\infty(\T, \cA_{-1})$, selfadjoint $\forall t$.
By Lemma \ref{lem:he}, we choose     $X_1 \in C^\infty(\T, \cA_0)$, selfadjoint  $\forall t$, s.t.
\begin{equation}
\label{hom.4}
\im [K_0,X_1(t)] + \pa_t X_1(t) = \fV(t) - \wh \fV(t)  \ , 
\end{equation}
where $ \wh \fV(t) $ is the averaged operator
(see \eqref{av.op}).
With this choice we have  
  \begin{align}
\label{4.1.3}
H^{+}(t)&= K_0 +  Z^{(1)}(t) + \fV^{(1)}(t) \ , 
\quad Z^{(1)}(t) := \wh \fV(t) \ .
\end{align} 
 By Lemma \ref{prop:av}, $Z^{(1)} \in  C^\infty(\T, \cA_0)$, it is selfadjoint $\forall t$, it  commutes with $\im \pa_t - K_0$.
The map $e^{- \im X_1(t)}$ fulfills \eqref{unitary2} thanks to Lemma \ref{MR}. This concludes the first step. 

 The iterative step $N\to N+1$ is proved
following the same lines, just adding the remark that $e^{\im
  X_{N+1}}Z^{(N)} e^{-\im X_{N+1}}-Z^{(N)}\in
C^\infty(\T, \cA_{-N-1})$, 
and solving the homological equation
\begin{equation}
\label{hom.5}
\im [K_0,X_{N+1}(t)] + \pa_t X_{N+1}(t) = \fV^{(N)}(t) - \wh{\fV^{(N)}}(t)  \ . 
\end{equation}
So one puts  $Z^{(N+1)} := Z^{(N)} + \wh{\fV^{(N)}}$. 
Note that $\wh{\fV^{(N)}} \in C^\infty(\T, \cA_{-N})$, so $ Z^{(N)}$ has an expansion in operators of decreasing order.

\end{proof}

It turns out that  property \eqref{Z.prop2} implies that 
$e^{\im t K_0} \, Z^{(N)}(t) \, e^{- \im t K_0}$ is time independent. A consequence of this fact is the following corollary. 
\begin{corollary}
\label{prop:rpdnf}
Consider equation \eqref{eq:v} with $\fV \in C^\infty(\T, \cA_0)$,  $\fV(t)$ selfadjoint $\forall t$.
 Fix $ N \in \N$ arbitrary. There  exists a change of coordinates 
 $\cU_N(t)$ unitary in $\cH$ and  fulfilling
\begin{equation}
\label{unitary}
\forall r \geq 0  \quad \exists c_r, C_r >0 \colon \qquad c_r \norm{\vf}_r  \leq \norm{\cU_N(t)^{\pm} \vf}_r \leq C_r \norm{\vf}_r , \qquad \forall t \in \R , 
\end{equation}
 such that  $\psi(t)$ is a solution of \eqref{eq:v} if and only if 
 $\phi(t) := \cU_N(t) \psi(t)$ solves 
\begin{equation}
\label{res.eq}
\im \pa_t \phi = \big(\la \fV \ra +  T_N + R_N(t) \big)\phi \ ; 
\end{equation}
here  $\la \fV\ra$ is the resonant average of $\fV$ (see  \eqref{res.av}),  $T_N \in \cA_{-1}$ is  time independent and selfadjoint and 
 $R_N \in C^\infty(\T, \cA_{-N})$.
\end{corollary}
\begin{proof}
Fix $N \in \N$ and apply Proposition  \ref{lem:rpnf2} to conjugate  equation \eqref{eq:v} to the form \eqref{eq.res}
via the change of variables \eqref{change.v1}.
Then we  gauge away $K_0$ by  the change of coordinates $\vf = e^{- \im t K_0} \phi$, getting 
$$
\im \pa_t \phi = e^{\im t K_0} \, \big( Z^{(N)}(t) + \fV^{(N)}(t) \big)  \, e^{- \im t K_0} \, \phi .
$$
Define
$$
\fH_N :=  e^{\im t K_0} \,  Z^{(N)}(t)   \, e^{- \im t K_0} , \qquad
R_N(t):=  e^{\im t K_0} \,  \fV^{(N)}(t)   \, e^{- \im t K_0} . 
$$
The operator $R_N\in C^\infty(\T, \cA_{-N})$ by Assumption II since $\fV^{(N)} \in C^\infty(\T, \cA_{-N})$. \\
Let us now prove that $\fH_N$ is time independent.
We know by Lemma \ref{lem:rpnf2} that $Z^{(N)}(t)$ commutes with $\im \pa_t  - K_0$; therefore
$$ 
\pa_t  \big( e^{\im t K_0} \,  Z^{(N)}(t)   \, e^{- \im t K_0} \big) 
= 
e^{\im t K_0} \, \big( \im [K_0,  Z^{(N)}(t)] + \pa_t Z^{(N)}(t) \big)   \, e^{- \im t K_0} = 0 
$$
and we get 
$$
\fH_N = e^{\im t K_0} \,  Z^{(N)}(t)   \, e^{- \im t K_0}\vert_{t = 0}  =  Z^{(N)}(0) \stackrel{\eqref{Z.exp}}{=}
\wh \fV(0) + T^{(N)}(0)  \stackrel{\eqref{av.W}}{=} \la \fV \ra + T^{(N)}(0)  . 
$$
So we put $T_N := T^{(N)}(0)$; clearly it belongs to  $\cA_{-1}$,   it is selfadjoint and time independent.

Finally we put $\cU_N(t) := e^{\im t K_0} \, e^{\im t X_N(t)} \, \cdots e^{\im t X_1(t)}$; estimate \eqref{unitary} follows from \eqref{unitary2} and Remark \ref{rem:K0.flow}.
\end{proof}

Coming  back to the original equation \eqref{eq.ab},  we 
apply Corollary \ref{prop:rpdnf} with  $\fV= V + W \in C^\infty(\T, \cA_{0})$, getting the following result:
\begin{corollary}
\label{cor:res.pseudo}
With the same assumptions of Theorem \ref{thm:ab}, the following holds true. 
 Fix $ N \in \N$ arbitrary. There  exists a change of coordinates 
 $\cU_N(t)$, unitary in $\cH$ and fulfilling \eqref{unitary}
 such that  $\psi(t)$ is a solution of \eqref{eq.ab} if and only if 
 $\phi(t) := \cU_N(t) \psi(t)$ solves 
\begin{equation}
\label{res.eq2}
\im \pa_t \phi = \big( \la V \ra + \la W \ra +  T_N + R_N(t) \big)\phi
\end{equation}
where  $T_N \in \cA_{-1}$ is selfadjoint and  time independent 
whereas  $R_N \in C^\infty(\T, \cA_{-N})$.
\end{corollary}

\subsection{Local energy decay estimates}
From now on we are going to assume also Assumption III.
In the previous section we have conjugated the original equation \eqref{eq.ab} to the resonant equation \eqref{res.eq2}.
In this section we  consider the effective equation obtained removing $R_N(t)$ from \eqref{res.eq2}, namely 
\begin{equation}
\label{eq.H}
\im \pa_t \vf = H_N \vf, \qquad H_N := \la V \ra + \la W \ra +  T_N,
\end{equation}
with $T_N \in \cA_{-1}$ of Corollary \ref{cor:res.pseudo}.
Note that $H_N$ is selfadjoint by Lemma \ref{prop:av} and Corollary \ref{cor:res.pseudo}.
The goal is to construct a  solution of \eqref{eq.H} with  polynomially in time growing Sobolev norms. 
Actually we will prove the following  slightly stronger result, namely the existence of a solution with {\em decaying  negative Sobolev norms}: 

\begin{proposition}[Decay of negative Sobolev norms]
\label{prop:decay}
With the same assumptions of Theorem \ref{thm:ab}, 
consider 
the  operator $H_N $  in \eqref{eq.H}.
For any  $ k  \in \N$, 
there exist a nontrivial solution $\vf(t) \in  \cH^k$ of \eqref{eq.H}   and   $\forall r \in [0, k]$ a constant $C_r >0$ such that 
\begin{equation}
\label{decay}
 \norm{ \vf(t)}_{{-r}}  \leq C_r \la t \ra^{-r} \,  \norm{\vf(0)}_r  \ , \qquad \forall t \in \R \ . 
\end{equation}
\end{proposition}

\begin{remark}
\label{rem:g}
 As $H_N$ is selfadjoint,  the conservation of the $\cH^0$-norm and 
Cauchy-Schwartz inequality give 
$$
\norm{\vf(0)}_0^2 = \norm{\vf(t)}_{0}^2 \leq \norm{\vf(t)}_{r} \ \norm{\vf(t)}_{{-r}}  \ , \qquad \forall t \in \R \ ,
$$
so that \eqref{decay} implies the growth of positive Sobolev norms:  
$$\norm{\vf(t)}_r \geq  \frac{1}{C_r} \frac{\norm{\vf(0)}_0^2}{\norm{\vf(0)}_r} \, \la t \ra^{r}  \ , \quad  \forall t \in \R \ . 
$$
\end{remark}

The rest of the section is devoted to the proof of Proposition \ref{prop:decay}.
As we shall see, it  follows  from a {\em local  energy decay estimate} for the operator $H_N$, namely a dispersive  estimate of the form 
\begin{equation}
\label{LEDE0}
\norm{\la A \ra^{-k} \, e^{- \im H_N t} \, g_J(H_N) \,  \vf}_0 \leq C_k \la t \ra^{-k} \norm{\la A \ra^{k} g_J(H_N) \vf}_0 \ , \qquad \forall t \in \R 
\end{equation}
where  $A \in \cA_1$,  $J \subset I_0$ is an interval  and  $g_J \in C^\infty_c(\R, \R_{\geq 0})$ with  $g_J \equiv 1$ on $J$. 

\begin{remark}
\label{rem:inf.many}
Actually estimate \eqref{LEDE0} show the existence of infinitely many solutions of \eqref{eq.H} with decaying negative Sobolev norms. In particular  this happens to any solution whose  initial datum $\vf(0)$ belongs to the  (infinite dimensional) set  ${\rm Ran }\, E_J(H_N)$, where $E_J(H_N)$ is the spectral projection of $H_N$ corresponding to the interval $J$.
\end{remark}

A possible approach  (which we will follow here) to obtain such estimate is  via  Sigal-Soffer minimal velocity estimates  \cite{Ski, GerSig, JenMouPer,HunSigSof, GNRS, Arbunich}. 
These estimates are based on  Mourre theory, let us recall this last one.

\paragraph{Mourre theory.} Let $\fH$ be a selfadjoint operator on the Hilbert space $\cH$, and denote by $\sigma(\fH)$ its spectrum. 
We further   denote by $\sigma_d(\fH)$ its discrete spectrum, $\sigma_{ess}(\fH)$  its essential spectrum, $\sigma_{pp}(\fH)$  its pure point spectrum, $\sigma_{ac}(\fH)$  its absolutely continuous spectrum and $\sigma_{sc}(\fH)$  its singular spectrum; 
 see e.g. \cite{ReedSimon} pag. 236 and  231 for their definitions. 
 Furthermore we denote by $E_\Omega(\fH)$  the spectral projection of $\fH$ corresponding to the Borel set $\Omega$ and by $m_\vf(\Omega):= \la E_{\Omega}(\fH) \vf, \vf\ra$ the spectral measure associated to $\vf \in \cH$.

Assume a selfadjoint operator  $\fA$ can be found such that 
$D(\fA)\cap \cH$ is dense in $\cH$. 
We put
\begin{equation}
\label{adjoint}
{\rm ad}^0_\fA(\fH) := \fH , 
\qquad {\rm ad}_\fA(\fH):= [\fH, \fA] , 
\qquad {\rm ad}^n_\fA(\fH) := [ {\rm ad}^{n-1}_\fA(\fH) , \fA] , \quad \forall n \geq 2 \ . 
\end{equation}
Consider the following properties:
\begin{itemize}
\item[(M1)] \label{M1} For some $\tN \geq 1$, the operators 	${\rm ad}^n_\fA(\fH)$ with $ n = 1, \ldots, \tN$,
can all be extended to bounded operators on $\cH$.
\item[(M2)] {\em Mourre estimate}: there exists an open interval $I\subset \R$ with compact closure and a function
$g_I \in C^\infty_c(\R, \R_{\geq 0})$ with  $g_I \equiv 1$ on $I$ such that 
\begin{equation}
\label{M2}
g_I(\fH) \, \im [\fH, \fA] \, g_I(\fH)  \geq \theta g_I(\fH)^2  + \fK
\end{equation}
for some $\theta >0$ and $\fK$ a selfadjoint compact operator on $\cH$.
\end{itemize}
If the estimate \eqref{M2} holds true with $\fK = 0$, we shall say that $\fH$ fulfills a {\em strict} Mourre estimate.\\
 Mourre theorem \cite{Mourre} 
says the following:
\begin{theorem}[Mourre] 
\label{thm:mourre}
Assume conditions   {\rm (M1)} -- {\rm (M2)} with $\tN = 2$.  In the interval $I$, the operator $\fH$ can have only absolutely continuous spectrum and finitely many eigenvalues of finite multiplicity. 
 If $\fK= 0$, there are no eigenvalues  in the interval $I$, i.e.  $\sigma(\fH) \cap I = \sigma_{ac}(\fH) \cap I$.
\end{theorem}

\begin{remark}
The version stated here of Mourre theorem is taken from \cite[Lemma 5.6]{Bach99} and   \cite[Theorem 4.7 -- 4.9]{CFKS}, and it has slightly weaker assumptions compared to \cite{Mourre}.
\end{remark}

\begin{remark}
Mourre theorem guarantees that $\sigma_{sc}(\fH)\cap I = \emptyset$ and, in case $\fK = 0$, $\sigma_{pp}(\fH) \cap I = \emptyset$.
However it does not guarantee that  $\sigma(\fH) \cap I \neq \emptyset$;  in our case we shall verify this property explicitly. 
\end{remark}

The key point is that if  $H_N$ fulfills  a {\em strict} Mourre estimate (namely with $\fK = 0$)  
then  one can prove  a  local energy decay estimate like \eqref{LEDE0} for the Schr\"odinger  flow of $H_N$. This is a quite general fact which follows exploiting minimal velocity estimates \cite{HunSigSof} and we prove it for completeness in  Appendix \ref{app:minimal}. 

So the next goal is to prove  that  $H_N$ satisfies a strict Mourre estimate over a certain interval $J \subset I_0$. 
During the proof we will use some standard results from functional calculus; we recall them in Appendix \ref{app:fc}. 
We shall also use the following lemma:
\begin{lemma}
\label{lem:ac.meas}
Let $\fH\in \cL(\cH)$ be selfadjoint. If $\lambda \in \sigma_{ac}(\fH)$, then  $\forall \delta >0$ one has
$$
\abs{[\lambda- \delta, \lambda + \delta] \cap \sigma(\fH)}>0 \ . 
$$
\end{lemma}
\begin{proof}
By contradiction, assume that $\exists \delta_0 >0$ such that 
$\abs{[\lambda- \delta_0, \lambda + \delta_0] \cap \sigma(\fH)} = 0$.
As $\lambda \in \sigma_{ac}(\fH)$, there exists $f \in \cH$ such that
$E_{[\lambda- \delta_0, \lambda + \delta_0] }(\fH) f \neq 0$ and the  spectral measure $m_f = \la E(\fH) f, f \ra$ is absolutely continuous.
Then
$$
 0 = m_f([\lambda- \delta_0, \lambda + \delta_0]) = 
 \la E_{[\lambda- \delta_0, \lambda + \delta_0] }(\fH) f, f \ra = \norm{E_{[\lambda- \delta_0, \lambda + \delta_0] }(\fH) f}_0^2 >0 
$$
giving a contradiction.
\end{proof}

\begin{lemma}
\label{lem:H.M2}
There exist $\epsilon_0, \tM >0$ such that, provided   $W$ fulfills \eqref{normW}, 
the following holds true:
\begin{itemize}
\item[(i)] 
There exists an interval $I \subset I_0$ such that
$\abs{ I \cap \sigma(H_N)} > 0$. 
\item[(ii)] $H_N$ fulfills a strict Mourre estimate over $I$: there exists   a function  $g_I \in C^\infty_c(\R, \R_{\geq 0})$ with ${\rm supp }\, g_I \subset I_0$,  $ g_I \equiv 1 $ on $I$, and $\theta' >0$ such that 
\begin{equation}
\label{pm}
g_I(H_N) \, \im [H_N, A] \, g_I(H_N) \geq \theta' g_I(H_N)^2   \ . 
\end{equation}
\end{itemize}
 Here $I_0$ is the interval and $A$ is the operator of Assumption {\rm III}.
\end{lemma}
\begin{proof}
During the proof we shall often use that  for $\fA, \fB, \fC \in \cL(\cH)$ and selfadjoints 
\begin{equation}
\label{op.ineq}
\begin{aligned}
&\fA \leq \fB \ \ \Rightarrow \ \  \fC \fA \fC \leq \fC \fB \fC , \qquad 
\norm{\fA}_{\cL(\cH)} \leq a \ \  \Rightarrow \ \  - a \leq \fA \leq a  \ . 
\end{aligned}
\end{equation}
 To shorten notation,   throughout the proof  we shall put
 $$H_0 := \la V \ra \ . $$
We split the proof in several steps.

\noindent {\sc Step 1:} By  Assumption III, 
$H_0$ fulfills a Mourre estimate over the interval $I_0$. The first step of the proof is to exhibit a subinterval  $I_1 \subset I_0$ containing only absolutely continuous spectrum of $H_0$, namely 
\begin{equation}
\label{I1}
\sigma(H_0) \cap I_1 = \sigma_{ac}(H_0) \cap I_1 \ ,  \qquad 
\abs{\sigma(H_0) \cap I_1}>0   \ , 
\end{equation}
and 
 over which $H_0$ fulfills a {strict} Mourre estimate:
 $\exists g_{I_1}\in C^\infty_c(\R, \R_{\geq 0})$, $g_{I_1} \equiv 1$ on $I_1$, ${\rm supp }\, g_{I_1} \subset I_0$, such that 
 \begin{equation}
\label{pm.s100}
g_{I_1}(H_0) \, \im [H_0, A] \, g_{I_1}(H_0)  \geq \frac{\theta}{2} g_{I_1}(H_0)^2    \ . 
\end{equation}
To prove this claim, first apply Mourre theorem 
to $H_0$ (note that    (M1) and (M2) are verified $\forall \tN 	\in \N$  by symbolic calculus and Assumption III), getting that   $\sigma(H_0) \cap I_0$ contains only  finitely many eigenvalues with finite multiplicity and absolutely continuous spectrum. 
In particular $|\bar{\sigma_{pp}(H_0)}\cap I_0|=0$ and by Assumption III $(i)$ it follows that $|\sigma_{ac}(H_0) \cap I_0| = |\sigma(H_0) \cap I_0|>0$. \\
So we take $\lambda_0 \in I_0 \cap (\sigma_{ac}(H_0) \setminus \sigma_{pp}(H_0))$ and a sufficiently small interval 
 $ I_1(\bar \delta) := (\lambda_0 - \bar\delta, \lambda_0 + \bar\delta) \subset I_0$,  $\bar \delta >0$, which does not contain  eigenvalues of $H_0$; this   is possible as the eigenvalues of $H_0$ in $I_0$ are finite.
Moreover  by Lemma \ref{lem:ac.meas},   $\abs{\sigma(H_0)\cap I_1( \delta)}>0$ for any $\delta >0$. 
Now take $\delta \in (0, \bar \delta)$ and a function  $g_{\delta} \in C^\infty_c(\R, \R_{\geq 0})$ with ${\rm supp } \, g_{\delta} \subset I_1(\delta)$ and $g_{\delta} = 1$ on $I_1(\frac{\delta}{2})$.
We claim that provided   $\delta \in (0, \bar \delta)$ is  sufficiently small 
\begin{equation}
\label{pm.s90}
\norm{g_{\delta}(H_0) Kg_{\delta}(H_0)}_{\cL(\cH)} \leq \frac{\theta}{2} 
\ ,
\end{equation}
where $\theta>0$ is the one of  Assumption III.
Indeed in   $I_1(\bar \delta)$ the spectrum of $H_0$ is absolutely continuous; this means that  $\forall \vf \in \cH$, the vector $\vf':= E_{I_1(\bar\delta)}(H_0) \vf$ belongs to the absolutely continuous subspace of $H_0$, namely its  spectral measure $m_{\vf'}$  is absolutely continuous w.r.t. the Lebesgue measure.
Now, since for any $\vf \in \cH$ one has by functional calculus  $g_\delta(H_0)=g_\delta(H_0) E_{I_1(\bar\delta)}(H_0)$, one has that 
$$
\norm{g_\delta(H_0) \vf}_0^2 = \norm{g_\delta(H_0) E_{I_1(\bar\delta)}(H_0) \vf}_0^2 = \int_\R g_\delta(\lambda)^2 \, \di m_{\vf'}(\lambda)  \to 0 \ \ \ \mbox{ as } \delta \to 0 
$$
by Lebesgue dominated convergence theorem.
In particular  $g_\delta(H_0)\to 0$ strongly  as $\delta \to 0$ and then, being  $K$  compact,
$g_\delta(H_0) K \to 0$ uniformly as $\delta \to 0$ (see e.g. \cite{Anselone}).
Therefore for  $ \delta \in (0, \bar \delta)$ sufficiently small   \eqref{pm.s90} holds true.\\
Using the assumption \eqref{mourre.ab},  \eqref{pm.s90} and  \eqref{op.ineq} we deduce that 
\begin{align*}
g_{\delta}(H_0)\, g_{I_0}(H_0) \, \im [H_0, A] \, g_{I_0}(H_0) \, g_{\delta}(H_0) 
& \geq  \theta g_\delta(H_0) \, g_{I_0}(H_0)^2 \, g_\delta(H_0)  - \frac{\theta}{2} \ ;
\end{align*}
next apply $g_\frac{\delta}{2}(H_0)$ to the right and left of the previous inequality, 
use again  \eqref{op.ineq} and the identity  $ g_{I_0}(H_0) \, g_\delta(H_0) \, 
g_{\frac{\delta}{2}}(H_0) = g_{\frac{\delta}{2}}(H_0) $ (which follows from  $g_{I_0} \, g_\delta\, 
g_{\frac{\delta}{2}} = g_{\frac{\delta}{2}} $),
to  get 
the strict Mourre estimate 
\begin{equation}
\label{pm.s101}
g_{I_1}(H_0) \, \im [H_0, A] \, g_{I_1}(H_0)  \geq \frac{\theta}{2} g_{I_1}(H_0)^2   
\end{equation}
where  $I_1:= I_1(\frac{\delta}{4})$ and $g_{I_1}:= g_{\frac{\delta}{2}}$ fulfills $g_{I_1} \equiv 1$ on $I_1$, ${\rm supp }\, g_{I_1} \subset I_1(\frac{\delta}{2})$.
Clearly $I_1$ fulfills \eqref{I1}.

\noindent {\sc Step 2:} We shall  prove that the selfadjoint operator
$$
H_{\la W \ra} := H_0 + \la W \ra 
$$
has a nontrivial spectrum in a subinterval $I_2  \subseteq I_1$, and over this interval it fulfills the  strict  Mourre estimate
\begin{equation}
\label{pm.s200}
g_{I_2}\big( H_{\la W \ra}\big)  \, \im [H_{\la W \ra}, A] \, g_{I_2}\big(H_{\la W \ra} \big) \geq \frac{\theta}{4} 
g_{I_2}\big( H_{\la W \ra}\big)^2 \ 
\end{equation}
for any $g_{I_2} \in C^\infty_c(\R, \R_{\geq 0})$ with ${\rm supp }\, g_{I_2} \subset I_1$, $g_{I_2} \equiv 1$ on $I_2$. 
To prove this, we exploit that  $\la W \ra \in \cA_0 $ is a small bounded perturbation of $H_0$, fulfilling, by \eqref{est.1}, \eqref{B-est2}
\begin{equation}
\label{pm.s20}
\exists M_0 \in \N , C_0 >0 \colon 
\qquad \norm{\la W \ra}_{\cL(\cH)} {\leq}
\, C_0 [W]_{M_0}  ,
\end{equation}
where we denoted
$$
[W]_{M} := \sup_{t \in \T} \, \wp^0_{M}(W(t)) \ . 
$$
First let us prove that $\sigma(H_{\la W \ra})\cap I_1 \neq \emptyset$.
Take again the same  $\lambda_0 \in \sigma(H_0) \cap I_1$ as in the previous step.
We claim that 
\begin{equation}
\label{pm.s21}
{\rm dist }\big(\lambda_0, \sigma(H_{\la W \ra})\big) \leq  C_0 \,  [W]_{M_0} . 
\end{equation}
If  $\lambda_0 \in \sigma(H_{\la W \ra})$ this is trivial. 
So assume that  $\lambda_0 $ belongs to the resolvent set of $ H_{\la W \ra}$.  As $\lambda_0 \in \sigma(H_0)$, 
by Weyl criterion   $\exists (f_n)_{n \geq 1} \in \cH$ with $\norm{f_n}_0 = 1$  such that
$\norm{(H_0 - \lambda_0) f_n}_0 \to 0$ as $n \to \infty$. Then  $\forall n \geq 1$ 
\begin{align*}
1 & = \norm{f_n}_0 = \norm{(H_{\la W \ra} - \lambda_0)^{-1} \, 
(H_{\la W \ra} - \lambda_0) f_n }_0 \leq \frac{1}{{\rm dist }\big(\lambda_0, \sigma(H_{\la W \ra})\big) } \norm{(H_{\la W \ra} - \lambda_0) f_n }_0\\
& \stackrel{\eqref{pm.s20}}{\leq } \frac{1}{{\rm dist }\big(\lambda_0, \sigma(H_{\la W \ra})\big) } \Big(  \norm{(H_0 - \lambda_0) f_n }_0 +   C_0 [W]_{M_0} \Big) 
\end{align*}
which proves \eqref{pm.s21} passing to the limit $n \to \infty$.
 Then, provided  $ [W]_{M_0}$ is sufficiently small, \eqref{pm.s21} implies that ${\rm dist }\big(\lambda_0, \sigma(H_{\la W \ra})\big) < \delta/8$. From this we learn that (recall $I_1 = (\lambda_0-\frac{\delta}{4}, \lambda_0 + \frac{\delta}{4})$)
 \begin{equation}
 \label{sH_W}
 \sigma(H_{\la W \ra}) \cap I_1 \neq \emptyset \ .
 \end{equation} 
Next we prove the Mourre estimate \eqref{pm.s200}; we shall work  perturbatively from  \eqref{pm.s100}.
First
 \begin{align*}
 g_{I_1}(H_0) \, \im [H_{\la W \ra}, A] \,  g_{I_1}(H_0) = 
 g_{I_1}(H_0) \, \im [H_0, A] \,  g_{I_1}(H_0) + 
 g_{I_1}(H_0) \, \im [\la W \ra, A] \,  g_{I_1}(H_0) ;
\end{align*}
we bound the first term in the right  hand side above from below using  \eqref{pm.s100}. Concerning the second term, we use
\begin{align}
\label{pm49}
\exists \, M_1 \in \N, \,  C_1>0  \colon \quad \norm{ \im [\la W \ra, A]}_{\cL(\cH)} \leq      C_1 [W ]_{M_1}   \ 
\end{align}
(by \eqref{est.1}, \eqref{est.3}, \eqref{B-est2}) 
and the inequalities \eqref{op.ineq} to bound it from above getting
$$
 g_{I_1}(H_0) \, \im [\la W \ra, A] \,  g_{I_1}(H_0) \geq - C_1 \, [W]_{M_1} \, 
 g_{I_1}(H_0)^2  \ . 
$$
Therefore we find
\begin{align}
\label{pm5}
g_{I_1}(H_0) \, \im [H_{\la W \ra}, A] \, g_{I_1}(H_0) \geq \left( \frac{\theta}{2} - C_1 [W]_{M_1} \right) \, g_{I_1}(H_0)^2  \ . 
\end{align}
Take  now  an open interval $I_2 \subset I_1$ such that  $\sigma(H_{\la W \ra}) \cap { I}_2 \neq \emptyset $ (it is possible by \eqref{sH_W}); take also 
$g_{I_2} \in C^\infty_c(\R, \R_{\geq 0})$ with ${\rm supp }\,  g_{I_2} \subseteq I_1$ and $g_{I_2} \equiv 1$ on $I_2$;
remark that  $
g_{I_1}  g_{I_2}= g_{I_2}$.
Now we  wish to replace $g_{I_1}(H_0)$ by $g_{I_2}(H_{\la W \ra})$ in \eqref{pm5}, thus getting the claimed estimate \eqref{pm.s200}. 
So write 
\begin{align}
\notag
g_{I_2}(H_{\la W \ra}) \, \im [H_{\la W \ra}, A] & \, g_{I_2}(H_{\la W \ra})
 = 
g_{I_2}(H_{\la W \ra}) \,  g_{I_1}(H_{\la W \ra}) \, \im [H_{\la W \ra}, A] \,  g_{I_1}(H_{\la W \ra}) \, g_{I_2}(H_{\la W \ra}) \\
 \label{pm6}
&  =  g_{I_2}(H_{\la W \ra}) \, g_{I_1}(H_0)\, \im [H_{\la W \ra}, A] \,  g_{I_1}(H_0) \, g_{I_2}(H_{\la W \ra}) \\
 \label{pm7}
&  \ +
g_{I_2}(H_{\la W \ra})\,   \Big(
  \big( g_{I_1}(H_{\la W \ra}) - g_{I_1}(H_{0}) \big) \, \im [H_{\la W \ra}, A] \, g_{I_1}(H_{0})  \\
 \label{pm8}
&  \ + 
  g_{I_1}(H_{\la W \ra})  \, \im [H_{\la W \ra}, A ] \,   \big( g_{I_1}(H_{\la W \ra}) - g_{I_1}(H_{0}) \big)
   \Big) \, g_{I_2}(H_{\la W \ra})
\end{align}
Again we estimate \eqref{pm6} from below and the other lines from above. First
\begin{align}
\label{pm80}
\eqref{pm6}  & \stackrel{ \eqref{pm5}}{\geq }
\left( \frac{\theta}{2} - C_{1} [W]_{M_1} \right) \, g_{I_2}(H_{\la W \ra})\, g_{I_1}(H_0)^2  \, g_{I_2}(H_{\la W \ra}) \ . 
\end{align}
We still have to bound from below $g_{I_2}(H_{\la W \ra})\, g_{I_1}(H_0)^2  \, g_{I_2}(H_{\la W \ra})$. 
To proceed  we  use that $g_{I_1}(H_{\la W \ra}) - g_{I_1}(H_0)$ is small in size, being bounded, via Lemma \ref{lem:g} and  \eqref{pm.s20}, by 
\begin{equation}
\label{pm81}
\begin{aligned}
\norm{g_{I_1}(H_{\la W \ra}) - g_{I_1}(H_0)}_{\cL(\cH)} \leq C \, [W]_{M_0}\, . 
\end{aligned}
\end{equation}
So   write
\begin{equation}
 \label{pm9}
\begin{aligned}
& g_{I_2}(H_{\la W \ra})\,  g_{I_1}(H_0)^2  \, g_{I_2}(H_{\la W \ra})
  = 
g_{I_2}(H_{\la W \ra})\, g_{I_1}(H_{\la W \ra})^2  \, g_{I_2}(H_{\la W \ra}) \\
 &  \ \  + g_{I_2}(H_{\la W \ra})\, 
  \Big( 
  g_{I_1}(H_{\la W \ra}) \,( g_{I_1}(H_0)- g_{I_1}(H_{\la W \ra})) \, + 
( g_{I_1}(H_0)- g_{I_1}(H_{\la W \ra})) \, g_{I_1}(H_0)  \Big) \, g_{I_2}(H_{\la W \ra}) . 
 \end{aligned} 
 \end{equation}
Therefore, using $
g_{I_1}  g_{I_2} = g_{I_2}$,    
estimates \eqref{pm81} and \eqref{op.ineq}, we deduce
$$
g_{I_2}(H_{\la W \ra})\, g_{I_1}(H_0)^2  \, g_{I_2}(H_{\la W \ra}) 
   \geq  \big(1- C [W]_{M_0} \big) \, g_{I_2}(H_{\la W \ra})^2 \ . 
$$
Thus we can finally  estimate line \eqref{pm6} from below using \eqref{pm80} and the previous estimate, concluding 
\begin{equation}
\label{pm10}
\eqref{pm6} \geq  
\left( \frac{\theta}{2}- C_{1}[W]_{M_1} \right)   \big(1- C [W]_{M_0} \big) \, g_{I_2}(H_{\la W \ra})^2   . 
\end{equation}
Next consider lines \eqref{pm7}, \eqref{pm8}. 
We use the bound  (see  \eqref{pm49})
$$
\norm{ [H_{\la W \ra}, A]}_{\cL(\cH^0)} \leq C \big( 1+ [W]_{M_1}\big) \ , 
$$
 and  \eqref{pm81} to get
  \begin{equation}
\label{pm11}
\eqref{pm7} + \eqref{pm8} \geq - C \, [W]_{M_0}\, (1+  [W]_{M_1}) \, g_{I_2}(H_{\la W \ra})^2 .
\end{equation}
Putting together \eqref{pm10} and \eqref{pm11} we finally find
$$
g_{I_2}(H_{\la W \ra}) \, \im [H_{\la W \ra}, A] \, g_{I_2}(H_{\la W \ra}) \geq  
\left(  \frac{\theta}{2}- C ([W]_{M_1} +  [W]_{M_0} +  [W]_{M_0}\,   [W]_{M_1})\right) \, 
g_{I_2}(H_{\la W \ra})^2 . 
$$
Thus, provided \eqref{normW} holds true for $\tM$ sufficiently large and $\epsilon_0$ sufficiently small, 
the strict Mourre estimate   \eqref{pm.s200}  follows.
Mourre theorem implies that  the spectrum of $H_{\la W \ra}$ in $I_2$ is absolutely continuous and by  \eqref{sH_W} it is also  nonempty;  summarizing (use also  Lemma \ref{lem:ac.meas})
\begin{equation}
\label{s.HW}
\sigma(H_{\la W \ra}) \cap I_2 = \sigma_{ac}(H_{\la W \ra}) \cap I_2
\qquad 
\mbox{and}
\qquad
\abs{\sigma(H_{\la W \ra}) \cap I_2} > 0 \ .
\end{equation}

\noindent{\sc Step 3:} The last step is to consider the  operator $H_N =  H_0 + \la W \ra + T_N = H_{\la W\ra} + T_N$, which, for the remaining part of the proof, we shall denote just by $H$. 
 We shall constantly use that any pseudodifferential operator of strictly negative order is a compact operator on $\cH$ (see  Remark \ref{rem:compact});  in particular $T_N \in \cA_{-1}$ is  compact.
 We begin by proving that
 \begin{equation}
 \label{I2.mes}
 \abs{\sigma(H) \cap I_2} > 0  \ . 
 \end{equation} 
Indeed by Weyl theorem
 $
 \sigma_{ess}(H) = \sigma_{ess}(H_{\la W \ra}) 
 $
and therefore
 \begin{align}
 \notag
 \sigma(H) \cap I_2
 & \supset   \sigma_{ess}(H ) \cap I_2 
   = \sigma_{ess}(H_{\la W \ra }) \cap I_2    = 
   \sigma(H_{\la W \ra}) \cap I_2   \ , 
 \label{pm.s299}
 \end{align}
since  $\sigma_d( H_{\la W \ra}) \cap I_2 = \emptyset$ having $H_{\la W \ra}$ no eigenvalues  in $I_2$. Then \eqref{I2.mes} follows by \eqref{s.HW}.

Next we prove that $H$ fulfills a Mourre estimate over  $ I_2$, i.e. 
\begin{equation}
\label{pm.s300}
g_{I_2}\big( H\big)  \, \im [H, A] \, g_{I_2}\big(H \big) \geq \frac{\theta}{4} 
g_{I_2}\big( H\big)^2 + K 
\end{equation}
with $K$ a compact operator.
We work perturbatively from \eqref{pm.s200}. Again first we compute
$$
g_{I_2}\big( H_{\la W \ra}\big)  \, \im [H, A] \, g_{I_2}\big(H_{\la W \ra} \big)  = 
g_{I_2}\big( H_{\la W \ra}\big)  \, \im [H_{\la W \ra}, A] \, g_{I_2}\big(H_{\la W \ra} \big)
+
g_{I_2}\big( H_{\la W \ra}\big)  \, \im [T_N, A] \, g_{I_2}\big(H_{\la W \ra} \big)
\ ;
$$
we estimate the first term in the r.h.s. above  by \eqref{pm.s200},  whereas the second term is  a compact operator since $[T_N, A] \in \cA_{-1}$. We obtain
\begin{equation}
\label{pm.12}
g_{I_2}\big( H_{\la W \ra}\big)  \, \im [H, A] \, g_{I_2}\big(H_{\la W \ra} \big) \geq \frac{\theta}{4} 
g_{I_2}\big( H_{\la W \ra}\big)^2  + K_1
\end{equation}
with $K_1$ a compact operator. 
Now we must replace $g_{I_2}\big( H_{\la W \ra}\big)$ with $g_{I_2}(H)$.
 We write 
\begin{align}
 \label{pm600}
 & g_{I_2} (H) \,   \im [H, A] \, g_{I_2}(H)
  = g_{I_2}(H_{\la W \ra})\, \im [H, A] \,  g_{I_2}(H_{\la W \ra}) \\
 \label{pm700}
&  \quad + 
  \big( g_{I_2}(H) - g_{I_2}(H_{\la W \ra}) \big) \, \im [H, A] \,g_{I_2}(H_{\la W \ra})  
  + 
  g_{I_2}(H)  \, \im [H, A ] \,   \big(g_{I_2}(H) - g_{I_2}(H_{\la W \ra}) \big) 
\end{align}
This time we use that $g_{I_2}(H) - g_{I_2}(H_{\la W 	\ra})$ is a compact operator, see  Lemma \ref{lem:g}.
Thus
\begin{align}
\notag
\eqref{pm600} 
&\stackrel{\eqref{pm.12}}{\geq}
\frac{\theta}{4}  \, g_{I_2}\big( H_{\la W \ra}\big)^2   + K_1
= \frac{\theta}{4} g_{I_2}(H)^2  + K_2
\label{pm900}
\end{align}
where $K_1$, $K_2$  are  compact operators.
Similarly, using that $\im [H, A] \in \cA_0$ is a bounded operator, we deduce that \eqref{pm700} is a  compact operator. Estimate \eqref{pm.s300} follows. \\
In particular $H$ is conjugated to $A$ over the interval $I_2$ fulfilling 
\eqref{s.HW}.
Proceeding as in Step 1, we produce  a subinterval $I \subset I_2$  such that
$$\abs{I \cap \sigma(H)}>0 \ , 
\qquad
I \cap \sigma(H)  = I \cap \sigma_{ac}(H)$$ and  over which $H$ fulfills the strict Mourre estimate    \eqref{pm}.
\end{proof}

The previous result has proved the existence of an interval $I$ over which $H_N$ fulfills a strict Mourre  estimate.
This implies that   $H_N$   fulfills  dispersive estimates in the form of local energy decay. 
In the literature there are various variants of this result, thus in Appendix \ref{app:minimal} we state and prove the one we apply here. 
\begin{corollary}
\label{cor:ledeH}
Fix $k \in \N$. For any  interval $J\subset I$, any  function $g_J \in C^\infty_c(\R, \R_{\geq 0})$ 
with ${\rm supp }\,  g_J \subset I$, 
$g_J \equiv 1$ 
on $J$, there exists a constant $C_k>0$ such that 
\begin{equation}
\label{eq:SS.app}
\norm{\la A \ra^{-k} \, e^{- \im H_N t} \, g_J(H_N) \,  \vf}_0 \leq C_k \la t \ra^{-k} \norm{\la A \ra^{k} g_J(H_N) \vf}_0 \ , \quad \forall t \in \R \ , \ \ \ \forall \vf \in \cH^k \ . 
\end{equation}
Moreover $J$ can be chosen so that  $\abs{J \cap \sigma(H_N)}> 0$ and  $\sigma(H_N) \cap J  = \sigma_{ac}(H_N)\cap J$.
\end{corollary}
\begin{proof}
Apply Theorem \ref{thm:SS}, noting that 
 condition (M1) at page \pageref{M1} is trivially satisfied $\forall n \in \N$ as $ {\rm ad}^n_A(H_N)\in \cA_0 \subset \cL(\cH)$, whereas the whole point of Lemma  \ref{lem:H.M2} was to verify (M2). 
 This gives estimate \eqref{eq:SS.app}. The right hand side is finite for $\vf \in \cH^k$ by Lemma \ref{lem:gJ.bd} below,  which ensures that $g_J(H_N) \vf \in \cH^k$. 
 Finally note that, since $\abs{I \cap \sigma(H_N)}>0$, it is certainly possible to choose  $J \subset I$ so that  $\abs{J \cap \sigma(H_N)}>0$; as $H_N$ fulfills a strict Mourre estimate over $I$, its spectrum in this interval is absolutely continuous, so the same is true  in $J$.
 
\end{proof}

\begin{lemma}
\label{lem:gJ.bd}
For any $k \in \N$,  $g_J(H_N)$ extends to a bounded operator $\cH^k\to \cH^k$. 
\end{lemma}
\begin{proof}
As
 $
 K_0^k \, g_J(H_N) \, K_0^{-k} = g_J(H_N) - [  g_J(H_N), K_0^k ]  K_0^{-k} ,
 $ 
it is clearly sufficient to show that 
 $[  g_J(H_N), K_0^k ] K_0^{-k}$ is bounded on $\cH$.
  The adjoint formula \eqref{ad2}
 gives 
  $$
[  g_J(H_N), K_0^k ]  K_0^{-k} = 
\sum_{j = 1}^k c_{k,j} \, {\rm ad }^j_{K_0}(g_J(H_N)) \, K_0^{-j}  \ ;
 $$
 then it is enough to show that 
  ${\rm ad }^j_{K_0}(g_J(H_N)) \in \cL(\cH)$.  As ${\rm ad }^j_{K_0}(H_N)$ is a bounded operator $\forall j$ (symbolic calculus), the result is an immediate application of Lemma \ref{ad.A.g(H)}.
\end{proof}
We finally prove  Proposition \ref{prop:decay}.
\begin{proof}[Proof of Proposition \ref{prop:decay}]
First we show that for any $k \in \N$, there exists $C_{2k} >0$ such that 
\begin{equation}
\label{disp.est0}
\norm{e^{- \im t H_N} g_J(H_N) \vf}_{-2k} \leq C_{2k} \, \la t \ra^{-2k} \, \norm{ g_J(H_N) \vf}_{2k} , \qquad \forall t \in \R , \ \ \ \forall \vf \in \cH^{2k} \ .
\end{equation}
This follows from 
Corollary \ref{cor:ledeH} with $k \leadsto 2k$. Indeed,   as $A \in \cA_1$, the operator $\la A \ra^{2k} = (1+ A^2)^k \in \cA_{2k}$ and therefore, by symbolic calculus,  $K_0^{-2k} \la A \ra^{2k}$ and $\la A \ra^{2k} K_0^{-2k}$ belong to $\cA_0 \subset \cL(\cH)$.
Then
\begin{align*}
 \norm{ e^{- \im t H_N} g_J(H_N) \vf}_{-2k}
& \leq 
\norm{ K_0^{-2k} \, \la A \ra^{2k}}_{\cL(\cH)} \, \norm{\la A \ra^{-2k}  e^{- \im t H_N}g_J(H_N) \vf}_0 \\
& \leq 
 C_{2k} \la t \ra^{-2k} \norm{\la A \ra^{2k} g_J(H_N) \vf}_0\\
&\leq C_{2k} \la t \ra^{-2k} \norm{\la A \ra^{2k} K_0^{-2k}}_{\cL(\cH)} \norm{g_J(H_N) \vf}_{2k}
\end{align*}
proving \eqref{disp.est0}.
Then   linear interpolation with the equality  
 $\norm{  e^{- \im t H_N} \vf_0}_0 = \norm{\vf_0}_0$ $\, \forall t$ gives
 $\forall r \in [0, 2k]$
$$
 \norm{  e^{- \im t H_N} g_J(H_N) \vf}_{-r}  \leq  C_r  \la t \ra^{-r} \norm{g_J(H_N) \vf}_r \ , \quad \forall t \in \R \ , \ \ \ \forall  \vf \in \cH^r \ . 
$$
Finally we must show that this estimate is not trivial, namely that  $\exists \vf \in \cH^k$ so that $g_J(H_N) \vf \neq 0$.
So 
take $J \subset I$ with  $\abs{J \cap \sigma(H_N)}> 0$ and  $\sigma(H_N) \cap J  = \sigma_{ac}(H_N)\cap J$, which is possible by Corollary \ref{cor:ledeH}.
As  $g_J(H_N)\cH \neq \{0\}$ and $\cH^k$ is dense in $\cH$, we have that   $g_J(H_N)\cH^{k} \neq \{0\}$. Then it is enough to take 
 $f \in \cH^{k}$ so that  $g_J(H_N)f \neq 0$, and put $\vf_0 := g_J(H_N)f $ which, 
 by Lemma \ref{lem:gJ.bd}, belongs to  
  $ \cH^{k}$. Such initial datum fulfills the claim of Proposition \ref{prop:decay}.
\end{proof}

\subsection{Proof of Theorem \ref{thm:ab}}
We are finally in position of proving Theorem \ref{thm:ab}.
Recall that in Corollary \ref{cor:res.pseudo} we have conjugated  equation \eqref{eq.ab} to  \eqref{res.eq2}
with a change of variables bounded $\cH^r \to \cH^r$ uniformly in time, whereas in Proposition \ref{prop:decay}
we have constructed a solution of the effective equation 
$\im \pa_t \psi = H_N \psi$ with decaying negative Sobolev norms, therefore with growing positive Sobolev norms.
The last step is to construct a solution of the full equation 
\eqref{res.eq2} with growing Sobolev norms. 
To achieve this, we exploit that the perturbation $R_N(t)$ is $N$-smoothing (Definition \ref{smoothing}).\\
So to proceed we fix the parameters.
First  fix   $r >0$, 
then choose  $N, k \in \N$ such that 
\begin{equation}
\label{Nk}
 N \geq 2r +2 , \quad k  	\geq  N - r . 
\end{equation}
Apply Corollary \ref{cor:res.pseudo}  with such $N$, producing the operators $T_N$, $R_N(t)$ and conjugating \eqref{eq.ab} to  \eqref{res.eq2}.
By Proposition 
\ref{prop:decay}, 
 $\exists \, \vf_0 \in \cH^k$ such that  $\vf(t) := e^{- \im t H_N} \vf_0$ fulfills  $\forall {\tt r} \in [0, k]$:
\begin{equation}
\label{decay2}
\norm{\vf(t)}_{- {\tt r}} \leq C_{{\tt r},N} \la t \ra^{-{\tt r} } \, \norm{\vf_0}_{{\tt r}}  , \qquad \forall t \in \R \ . 
\end{equation}
We look for an 
 exact solution $\phi(t)$  of \eqref{res.eq2} of the form  $\phi(t) = \vf(t) + u(t)$, i.e.   $u(t)$  has to  satisfy
\begin{equation}
\label{}
\im \pa_t u = \big(H_N + R_N(t) \big) u  + R_N(t) \vf(t) .
\end{equation}
Denoting by $U_N(t,s)$ the linear propagator of $H_N + R_N(t)$, we choose 
\begin{equation}
\label{u}
u(t) : = \im  \int\limits_t^{+ \infty} U_N(t,s) \, R_N(s) \, \vf(s) \, \di s . 
\end{equation}
We estimate the  $\cH^r$ norm of $u(t)$. 
As 
$$
\sup_t \norm{[ H_N + R_N(t), \, K_0] }_{\cL(\cH^m)} <  C_m < \infty  \ , \qquad \forall m \in \R , 
$$
Theorem 1.5  of \cite{MaRo} guarantees
  that   the  propagator $U_N(t,s)$ extends to a bounded operator $\cH^r \to \cH^r$  fulfilling\footnote{apply the theorem with $\tau = 0$ and note that in that paper we defined  $\norm{\psi}_r \equiv \norm{K_0^{r/2} \psi}_0$, therefore the estimate in that paper reads  explicitly $\norm{K_0^{r/2}U_N(t,s)\psi }_{0} \leq C_r \, \la t-s \ra^{r/2} \norm{K_0^{r/2} \psi}_0$
  } 
 $$
\forall r >0 \ \ \ \exists \, C_r >0 \colon \qquad \norm{U_N(t,s)}_{\cL(\cH^r)} \leq C_r \, \la t-s \ra^{r} , \quad \forall t,s \in \R \ . 
 $$
 This estimate, the smoothing property $R_N(t)\colon  \cH^{r-N} \to \cH^{r}$ and  \eqref{decay2} with ${\tt r} := N-r \in [0,k]$ give
 \begin{align*}
 \norm{u(t)}_r 
 &
 \leq C_r  \int\limits_t^{+ \infty}   \la t-s \ra^r \norm{R_N(s) \, \vf(s)}_r \, \di s 
 \leq
 C_r  \int\limits_t^{+ \infty}   \la t-s \ra^r  \, \norm{\vf(s)}_{-(N-r)} \, \di s\\
 & \leq C_{r,N}  \, \norm{\vf_0}_{N - r}  \int\limits_t^{+ \infty}   \la t-s \ra^r  \, \frac{1}{\la s \ra^{N-r}} \,  \, \di s
 \leq 
  C_{r,N}  \, \norm{\vf_0}_{k}   \la t \ra^{-1} \ .
 \end{align*}
In particular the $\cH^r$ norm of  $u(t)$ decreases to 0 as $t \to \infty$. 
Then  $\phi(t) = \vf(t) + u(t)$ fulfills 
\begin{equation}
\label{}
\norm{\phi(t) }_r \geq \norm{\vf(t)}_r - \norm{u(t)}_r \geq 
c_r \frac{\norm{\vf_0}_0^2}{\norm{\vf_0}_r} \la t \ra^r - 
 C_{r,N} \norm{\vf_0}_k \la t \ra^{-1} \geq C \la t \ra^r \ ,  \quad \forall |t|  \geq T  \, , 
\end{equation}
where we used \eqref{decay2} with ${\tt r} = r$ and   Remark \ref{rem:g}. \\
Finally we get a solution of the original equation 
\eqref{eq.ab} putting 
$\psi(t) = \cU_N(t)^{-1} \phi(t)$, recall Proposition \ref{prop:rpdnf}. The operator  $\cU_N(t)$ fulfills  \eqref{unitary}, thus  $\psi(t)$  has polynomially growing Sobolev norms as \eqref{eq:gr},  concluding the proof of  Theorem \ref{thm:ab}.
 
 We can also prove the existence of infinitely many solutions undergoing growth of Sobolev norms.
\begin{corollary}
\label{cor:inf.many}
There are infinitely many distinct solutions of equation \eqref{eq.ab} with growing Sobolev norms.
\end{corollary} 
\begin{proof}
We fix $r>0$ and choose $N,k$ as in \eqref{Nk}.
From the previous proof, it follows that any initial data of the form
$$
\psi(0) := (\uno + \cK_0) \vf  \ , \qquad \cK_t\vf:=  \im \int_t^{+ \infty} U_N(t,s) R_N(s) e^{- \im s H_N}\vf  \, \di s  ,  \, \qquad t \geq 0 , 
$$
with $\vf \in {\rm Ran }\, g_J(H_N) \cap \cH^k$, gives rise to a solution with growing Sobolev norms (see also Remark \ref{rem:inf.many}).  Here $J$ is the interval of Corollary \ref{cor:ledeH}.
In particular, as 
$\abs{J \cap \sigma(H_N)}> 0$ and  $\sigma(H_N) \cap J  = \sigma_{ac}(H_N)\cap J$, the set ${\rm Ran }\, g_J(H_N)$ has infinite dimension.
Let us prove that $\uno + \cK_0$ is injective. 
Assume 
there are $\vf_1 \neq \vf_2 \in {\rm Ran }\, g_J(H_N)\cap \cH^k$ with  $(\uno + \cK_0)\vf_1 = (\uno + \cK_0) \vf_2$.
Put $u_j(t) : =\cK_t \vf_j$, $j = 1,2$; 
arguing as in the previous proof one has 
$\norm{u_j(t)}_r \to 0$ as $t \to \infty$.\\
Then   $\cU_N(t)^{-1}(e^{- \im t H_N}\vf_j + u_j(t))$, $j =1,2$, both solve \eqref{eq.ab} and have the same initial datum, so they are the same solution $\psi(t)$ of equation \eqref{eq.ab}.
Then 
\begin{align*}
\norm{\vf_1 - \vf_2}_0 &
 = \norm{e^{- \im t H_N} (\vf_1 - \vf_2)}_0 \\
&  \leq C_r \norm{\cU_N^{-1}(t) e^{- \im t H_N}(\vf_1 - \vf_2)}_r  \leq C_r \big(\norm{u_1(t)}_r + \norm{u_2(t)}_r\big) \to 0 
\end{align*}
as $t \to \infty$. 
Hence $\vf_1 = \vf_2$.
\end{proof}

\section{Applications}
\label{sec:app}
In the following section we apply Theorem \ref{thm:ab} to  the harmonic oscillator on $\R$ and  the half-wave equation on $\T$.
In both  cases we construct transporters which are stable under small, time periodic, pseudodifferential perturbations.

\subsection{Harmonic oscillator on $\R$} 
Consider the quantum harmonic oscillator 
\begin{align}
\label{har.osc}
& \im \pa_t \psi = \frac{1}{2} ( - \pa_x^2 + x^2) \psi +  V(t, x, D)\psi  , \quad x \in \R  .  
\end{align}
Here $K_0 := \frac{1}{2} \left(-\partial_{x}^2 + x^2\right)$ is the quantum  Harmonic oscillator, the scale of  Hilbert spaces is defined  as 
usual by  $\cH^r = {\rm Dom }\left(K_0^r\right)$, and the base space $(\cH^0, \la \cdot, \cdot \ra)$ is $L^2(\R, \C)$   with its  standard  scalar product.
The perturbation  $V$ is chosen as the Weyl quantization of a  symbol belonging to the following class:
\begin{definition}
\label{symbol.ao2}
A function $f$ is a {\em  symbol of order } $\rho \in\R$ if  $f \in C^\infty(\R_x \times \R_\xi, \C)$ and 
               $\forall \alpha, \beta \in \N_0$, there exists $C_{\alpha, \beta} >0$ such that
$$
 \vert \partial_x^\alpha \, \partial_\xi^\beta f( x,\xi)\vert \leq C_{\alpha,\beta} \ (1 + |x|^2 + |\xi|^2)^{\rho-\frac{\beta + \alpha}{2}}  \ . 
$$
             We will write $f \in S^\rho_{\hos}$.
\end{definition}
   We endow $S^\rho_\hos$  with the family of seminorms
\begin{equation*}
\wp^\rho_j(f) := \sum_{|\alpha| + |\beta| \leq j}
\ \ \sup_{(x, \xi) \in \R^{2}} \frac{\left|\partial_x^\alpha \, \partial_\xi^\beta f(x,\xi)\right|}{ \left(1 + |x|^2 + |\xi|^2\right)^{\rho-\frac{\beta +\alpha}{2}} } \ , \qquad j \in \N \cup \{ 0 \} \ . 
\end{equation*}
Such seminorms turn $S^\rho_\hos$ into a Fr\'echet space.
If a symbol $f$ depends on additional parameters (e.g. it is time dependent), we ask that all the seminorms  are uniform w.r.t. such parameters.\\
 To a symbol $f \in S^\rho_\hos$ we associate the operator
$f(x, D)$  by standard Weyl quantization
 
       $$
   \Big( f(x, D) \psi \Big)(x) := \frac{1}{2\pi} \iint_{y, \xi \in \R} {\rm e}^{\im (x-y)\xi} \, f\left( \frac{x+y}{2}, \xi \right) \, \psi(y) \, \di y \di \xi \ . 
        $$


\begin{definition}
\label{pseudo.an}
  We say that $F\in \cA_\rho$ if it is a pseudodifferential operator
  with symbol of class $S^\rho_{\hos}$, i.e., if there
  exists $f \in S^\rho_{\hos}$ and $S$ smoothing (in the sense of Definition \ref{smoothing})  such that $F = f(x, D_x)+ S$.
\end{definition}

\begin{remark} With our numerology, the symbol of the harmonic oscillator $K_0$ is of order 1,  $\frac12({x^2 + \xi^2}) \in S^1_{\hos}$, and not of order 2 as typically in the literature.
\end{remark}

As an application of the abstract theorems, we
describe a class of operators  which are transporters.
This class, which we call {\em smooth T\"oplitz operators},  is easily described in terms of their matrix elements, which we now introduce.
We denote by  $\{ \be_n \}_{n \in \N}$  the Hermite basis, formed by the  (orthonormal) eigenvectors of the Harmonic oscillator $K_0$:
\begin{equation}
\label{har.K0}
K_0 \be_n = \left( n-\frac{1}{2}\right) \be_n , \quad 
 \norm{\be_n}_0 = 1 ,  \quad   n \in \N \ .
\end{equation}
To each operator $\fH \in \cL(\cH)$ we   associate its {\em matrix}
$(\fH_{mn}) _{m,n \in \N}$ with respect to the Hermite basis,  whose elements are given by 
\begin{equation}
\label{matrix}
\fH_{mn}:=  \la \fH \, \be_n, \be_m \ra \ , \qquad \forall m,n \in \N \ . 
\end{equation}

\begin{remark}
\label{rem.m.s}
If $\fH$ is selfadjoint, so is its matrix $(\fH_{mn}) _{m,n \in \N}$, in particular $\fH_{mn} = \bar{\fH_{nm}}$.
\end{remark}

\begin{definition}[Smooth T\"oplitz operators]
\label{smoothT}
 A linear operator $\fH \in \cL(\cH)$ is said a {\em  T\"oplitz operator} if the entries of its  matrix are constant along each diagonal, i.e. 
\begin{equation}
\label{smoothT.1}
 \fH_{m_1 n_1} = \fH_{m_2 n_2} , \quad \forall  m_1, n_1, m_2, n_2 \in \N  \colon \ \ \ m_1 - n_1 = m_2-  n_2 \ . 
\end{equation}
 A T\"oplitz operator is said {\em smooth} if its matrix elements decay fast  off diagonal, i.e. $\forall N > 0$, $\exists C_N >0$ such that
\begin{equation}
 \label{smoothT.2}
 \abs{\fH_{m n} } \leq \frac{C_N}{\la m-n\ra^N} \ ,  \qquad \forall m, n \in \N \  . 
\end{equation}
\end{definition}

\begin{example}  The shift operators 
$S$ and its adjoint $S^*$ are defined on the Hermite functions $\{\be_n\}_{n \geq 1}$ by
\begin{equation}
\label{def.S}
S \be_n  = \be_{n+1} \ , \quad \forall n \in \N \ , 
\qquad 
S^* \be_n = 
\begin{cases}
0 &  \mbox{ if } n = 1 \\
\be_{n-1} & \mbox{ if } n \geq 2 
\end{cases} \ . 
\end{equation}
 The action of $S$ (and  of $S^*$) is extended on all $\cH$ by linearity, giving
$S \psi =  \sum_{n \geq 1} \psi_{n}  \be_{n+1}$, where we defined  $\psi_n := \la \psi, \be_n \ra$ for $n \geq 1$. 
Their matrices are given by
$$
(S_{mn} )_{m,n \in \N} = 
 \begin{pmatrix}
0 &  &     &  \\
1 & 0 &   &  \\
  &  1 & 0  & \\
  &      &  \ddots  &\ddots 
\end{pmatrix}
, 
\qquad
(S^*_{mn} )_{m,n \in \N} = 
 \begin{pmatrix}
0 & 1 &     &  \\
 & 0 & 1  &  \\
  &   & 0  & 1\\
  &      &   &\ddots 
\end{pmatrix}
, 
$$
from which it is clear that both $S$ and $S^*$  are smooth T\"oplitz operators.
\end{example}

We prove in the following  that any smooth T\"oplitz operator is actually a pseudodifferential operator in $\cA_0$, see Lemma \ref{top.pse}.

As an application of the abstract theorems, we show that any smooth T\"oplitz operator becomes a transporter for the Harmonic oscillator once it is multiplied by an appropriate scalar time periodic function.

\begin{theorem}
\label{thm:har0}
Let  $ \fV(x, D) $ be a  selfadjoint and smooth T\"oplitz operator  (see Definition \ref{smoothT}). Take $m,n \in \N$,   $m > n $,  such that
the matrix element
$$
\fV_{m-n}:= \la \fV(x, D) \, \be_n, \be_m \ra \neq 0 \ . 
$$
Then 
\begin{equation}
\label{V.har0}
V(t,x, D):= \cos((m-n)t) \, \fV(x, D)
\end{equation}
 is a transporter for 
\eqref{har.osc}. More precisely, $\forall r \geq 0$  there exist  a
solution $\psi(t) \in \cH^r$ of \eqref{har.osc} and constants $C, T >0$ such that 
$$
\norm{\psi(t)}_r \geq C \la t \ra^r , \quad \forall t > T .
$$
\end{theorem}

The theorem follows applying  Theorem \ref{thm:ab0}. So we check that  Assumptions I-III are fulfilled. 
Regarding Assumption I,  it is the usual Weyl calculus for  symbols in $S^\rho_{\hos}$, see e.g. \cite{Shubin}.
Concerning  Assumption II,  one has   $\sigma(K_0) = \{ n - \frac12 \}_{n \in \N}$. Furthermore   Egorov theorem for the Harmonic oscillator \cite{ho2} states that  the map  $t \mapsto  e^{-\im t K_0} \fA e^{ \im t K_0} \in C^\infty(\T, \cA_\rho)$ for any $\fA \in \cA_\rho$ (use also the periodicity of the flow of $K_0$).
This can be seen e.g. by remarking that the   symbol of $e^{-\im t K_0} \fA e^{ \im t K_0} $  is  
$a\circ \phi_{\hos}^t$, where $a \in S^\rho_\hos$ is the symbol of $\fA$ and  $\phi^t_{\hos}$ is the time $t$ flow of the harmonic oscillator; explicitly 
\begin{align}
\label{symb}
\left(a\circ \phi^t_{\hos}\right)(x, \xi) = a(x \cos t + \xi \sin t, -x \sin t + \xi \cos t) \ .
\end{align}
\noindent{\bf Verification of Assumption III.} 
First we  show that smooth T\"oplitz  operators belong to $\cA_0$.   We  exploit  Chodosh's characterization \cite{chodosh}, which we now recall.
 Define  the discrete difference operator $\triangle$ on a function $M \colon \N \times \N \to \C$ by
$$
(\triangle M)(m,n) := M(m+1, n+1) - M(m,n)  \ ,
$$
and its powers $\triangle^\gamma$, $\gamma \in \N$,  by  $\triangle$ applied $\gamma$-times.
\begin{definition}[Symbol matrix]
A  function $M \colon \N \times \N \to \C$ will be said to be a {\em symbol matrix of order $\rho$} if for any $\gamma \in \N_0$, $ N \in \N$, there exists $C_{\gamma, N} >0$ such that 
\begin{equation}
\label{symb.mat}
 \abs{(\triangle^\gamma M)(m,n)} \leq C_{\gamma, N} \frac{(1 + m + n)^{\rho-|\gamma|}}{\la m-n \ra^N}  \ , \quad \forall m,n \in \N  \ . 
\end{equation}
\end{definition}
The connection between pseudodifferential operators  and symbol matrices is given by Chodosh's characterization:
\begin{theorem}[\cite{chodosh}]
\label{thm:chod}
An operator $\fH$ belongs to $\cA_\rho$ if and only if  its matrix  $M^{(\fH)}(m,n):= \fH_{mn}$ (as defined in \eqref{matrix})   is a symbol matrix of order $\rho$.
\end{theorem}

As a direct consequence we have the following result:
\begin{lemma}
\label{top.pse}
Any smooth T\"oplitz operator is a pseudodifferential operator in $\cA_0$.
\end{lemma}
\begin{proof}
We use Theorem \ref{thm:chod}. Let $\fH$ be smooth T\"oplitz and 
 put   $M^{(\fH)}(m,n):= \fH_{mn}$.
Then \eqref{symb.mat} holds with $\rho = \gamma = 0$ by \eqref{smoothT.2}. By \eqref{smoothT.1} one has 
$\triangle M^{(\fH)} = 0 $;  so \eqref{symb.mat} holds also $\forall \gamma \geq 1$.
\end{proof}

In particular  $V(t, x, D) = \cos((m-n)t) \fV(x, D)$ belongs  to $C^\infty(\T, \cA_0)$, which is the first required property of Assumption III. 
\begin{remark}
\label{rem:Sk}
The shift operators $S, S^*$, defined in \eqref{def.S},  belong to $\cA_0$ being smooth T\"oplitz. Also their (integer) powers 
$S^k$, $S^{*k}$, given for $k \in \N$ by 
\begin{equation}
S^k \be_n  = \be_{n+k} \ , \quad \forall n \in \N \ , 
\qquad 
S^{*k} \be_n = 
\begin{cases}
0 &  \mbox{ if } n \leq  k \\
\be_{n-k} & \mbox{ if } n \geq k+1  
\end{cases} \ 
\end{equation}
are smooth T\"oplitz, so in $\cA_0$.
\end{remark}

Next we compute the  resonant average of $V(t, x, D)$.

\begin{lemma}
\label{hw.HO3}
Let $V(t,x, D)$ as in \eqref{V.har0}. Its  resonant average $\la V \ra$ (see \eqref{res.av}) is 
\begin{equation}
\label{H0.har}
\la V \ra =  \frac12  \left(\fV_{k}\, S^k + \bar{\fV_k} \, S^{*k} \right) \ , \qquad k:= m - n  \in \N  \ , 
\end{equation}
where  $S \in \cA_0$ is defined in \eqref{def.S} and $\fV_k := \fV_{m-n}:= \la \fV \, \be_n, \be_m \ra \in \C$.
\end{lemma}
\begin{proof}
For $\ell \in \N$, denote by 
$\Pi_\ell \vf := \la \vf, \be_\ell \ra \, \be_\ell$  
the projector on the Hermite function $\be_{\ell}$.
 Clearly 
$$
e^{\im s K_0} \, \Pi_\ell   =  \Pi_\ell  \, e^{\im s K_0}  =  e^{\im s (\ell - \frac12)}  \, \Pi_\ell , \quad \forall \ell \in \N \ . 
$$
From now on we simply write $\fV\equiv \fV(x,D)$. Using this identity and writing  $\uno = \sum_{\ell \geq 1} \Pi_\ell$  we get 
\begin{align*}
e^{\im s K_0} \, \fV \, e^{-\im s K_0} 
&  = \sum_{j , \ell \geq 1}  e^{\im s(j - \ell)} \Pi_j \, \fV \, \Pi_\ell =  \sum_{j , \ell \geq 1}  e^{\im s(j - \ell)} \, \la \cdot , \be_\ell \ra \, \la \fV \be_\ell, \be_j \ra \, \be_j  \ . 
\end{align*}
Now we compute, with $k := m-n \in \N$, 
\begin{align*}
\la V \ra & = \frac{1}{2\pi} \int_0^{2\pi} \cos(k s) \, 
e^{\im s K_0} \, \fV \, e^{-\im s K_0} \di s
 = \sum_{j , \ell \geq 1}  \la \fV \be_\ell, \be_j \ra  \, \la \cdot , \be_\ell \ra \,  \be_j  \,\frac{1}{2\pi} \int_0^{2\pi} \cos(k s) \, e^{\im s(j - \ell)} \, \di s \\
 & = \frac12 \sum_{\ell \geq 1} \la \fV \be_\ell, \be_{\ell + k} \ra   \la \cdot , \be_\ell \ra \,  \be_{\ell +k} + 
\frac12  \sum_{\ell \geq k+1}  \la \fV \be_\ell, \be_{\ell - k} \ra  \la \cdot , \be_\ell \ra \,  \be_{\ell -k } 
  =
\frac12  \fV_k \, S^k + \frac12\bar{\fV_k} \, S^{*k} 
\end{align*}
where in the last line we used  $\fV_{-k} = \la \fV \be_\ell, \be_{\ell - k} \ra  = \bar{\la \fV \be_{\ell-k}, \be_{\ell} \ra } = \bar{\fV_k}$ being $\fV$ selfadjoint and smooth T\"oplitz  (see Remark \ref{rem.m.s}).
\end{proof}

Now define the selfadjoint operator 
\begin{equation}
\label{A.ho2}
A := \frac{\fV_k}{\im} \, (K_0 + \frac12 ) \, S^k -
 \frac{\bar{\fV_k}}{\im} \, S^{*k} \,  (K_0 + \frac12 ) 
 -  \frac{\bar{\fV_k}}{\im} \, (K_0 + \frac12 ) \, S^{*k}
+  \frac{{\fV_k}}{\im} \, S^k \,  (K_0 + \frac12 ) \ ,
\end{equation}
which belongs to  $\cA_1$ by symbolic calculus as $K_0 \in \cA_1$ and $S, S^* \in \cA_0$ (see Remark \ref{rem:Sk}).\\
The next lemma verifies Assumption III.
\begin{lemma}
\label{lem:har.mourre}
Assume that $\fV_k \neq 0$. 
The following holds true:
\begin{itemize}
\item[(i)] The spectrum of the operator $H_0:= \la V \ra$ fulfills $\sigma(H_0)  \supseteq \left[- |\fV_k|,  \,  |\fV_k|   \right]$.
\item[(ii)] 
For any interval  $I_0 \subset  \left[- |\fV_k|,  \,  |\fV_k|   \right]$, any $g_{I_0} \in C^\infty_c(\R, \R_{\geq 0})$ with $g_{I_0} \equiv 1$ over $I_0$
and ${\rm supp }\, g_{I_0} \subset  \left[- |\fV_k|,  \,  |\fV_k|   \right] $, there exist $\theta >0$  and $\fK$ compact operator such that 
$$
g_{I_0}(H_0) \, \im [H_0, A] \, g_{I_0}(H_0) \geq \theta \,  g_{I_0}(H_0)^2  + \fK  \ .
$$
Here  $A$ is defined in \eqref{A.ho2}.
\end{itemize}
\end{lemma}
\begin{proof}
$(i)$ 
Let $\tf(\rho):= {\rm Re}(\fV_k \, e^{- \im \rho k})$.
We shall prove that $\tf(\rho) \in \sigma(H_0)$  $\forall \rho \in \R$, from which the claim follows. 
As $H_0$ is selfadjoint,  it is enough to construct a Weyl sequence for $\tf(\rho)$, i.e. a sequence $(\psi^{(n)})_{n \geq 1}$ with $\norm{\psi^{(n)}}_0 = 1$ $\, \forall n$ and $\norm{(H_0  - \tf(\rho)) \psi^{(n)}}_0 \to 0$ as $n \to \infty$.
We put 
$$
\psi^{(n)} := \frac{1}{\sqrt n} \sum_{\ell = 1}^n e^{\im \rho \ell} \be_\ell \ .
$$
Then  $\norm{\psi^{(n)}}_0 = 1$ $\, \forall n$ and a direct computation shows that for $n > k$
$$
H_0 \psi^{(n)} = 
\frac{1}{\sqrt{n}} \frac{\bar \fV_k}{2} e^{\im \rho k} \sum_{m=1}^k e^{\im \rho m} \, \be_m 
+
\frac{1}{\sqrt{n}}\tf(\rho)  \sum_{m=k+1}^{n-k}  e^{\im \rho m} \, \be_m 
+
\frac{1}{\sqrt{n}} \frac{ \fV_k}{2} e^{-\im \rho k} \sum_{m=n-k+1}^{n+k} e^{\im \rho m} \, \be_m  \ .
$$
Thus one finds a constant $C_k >0$ such that
$$
\norm{(H_0  - \tf(\rho)) \psi^{(n)}}_0 \leq \frac{C_k}{\sqrt n} \to 0  \quad \mbox{ as }  n \to \infty  \ , 
$$
proving that $\psi^{(n)}$ is a Weyl sequence; by Weyl criterium $ \tf(\rho) \in \sigma(H_0)$.

\noindent $(ii)$  First note that, by \eqref{har.K0} and \eqref{def.S}, 
one has  $\forall k \in \N$ 
\begin{align}
\label{comm.id.har}
[S^k, K_0] = - k S^k , \qquad [S^{*k}, K_0] = k S^{*k} , \qquad [S^{*k}, S^k ] = \Pi_{\leq k }  \\
S^k S^{*k} = \uno - \Pi_{\leq k} \, \qquad S^{*k} S^k  = \uno 
\end{align}
where $\Pi_{\leq k}:= \sum_{\ell = 1}^k \Pi_\ell$ is the projector on the Hermite modes with index $\leq k$. 
Using \eqref{comm.id.har}  a direct computation gives 
\begin{align*}
\im [H_0 , A ] & = 	k\big( 2 |\fV_k|^2 - \fV_k^2 S^{2k}  - \bar{\fV}^2_k S^{*2k} - |\fV_k|^2  \Pi_{\leq k} \big) + 2|\fV_k|^2 (K_0 + \frac12) \Pi_{\leq k} \\
&
= 		4k \big(|\fV_k|^2 - H_0^2\big)  +  2|\fV_k|^2 (K_0 + \frac12 - k) \Pi_{\leq k} \ . 
\end{align*} 
Clearly $\fK:= 2|\fV_k|^2 (K_0 + \frac12 - k) \Pi_{\leq k} $ is  compact, being finite rank. \\
Next put $\tilde f(\lambda) = 4k (|\fV_k|^2 - \lambda^2)$ 
getting  $\forall \vf \in \cH$ 
\begin{equation}
\label{mourre.har1}
\la g_{I_0}(H_0) \, \im [H_0, A]  \,  g_{I_0}(H_0) \vf, \vf \ra = \la g_{I_0}(H_0) \, \tilde f(H_0) \,  g_{I_0}(H_0) \vf, \vf \ra + \la g_{I_0}(H_0) \, \fK  \,  g_{I_0}(H_0) \vf, \vf \ra \ . 
\end{equation}
Note that $\tilde f$ is strictly positive in the interior of $ \left[- |\fV_k|,  \,  |\fV_k|   \right]$; 
we put
$$
\theta := \inf \{ \tilde f(\lambda) \colon \lambda \in {\rm supp }\, g_{I_0} \} > 0 \ .
$$
With this information we apply the spectral theorem and get
\begin{align}
\notag
\la g_{I_0}(H_0) \, \tilde f(H_0) \,  g_{I_0}(H_0) \vf, \vf \ra
& = \int\limits_{\lambda \in \sigma(H_0)} g_{I_0}(\lambda)^2 \, \tilde f(\lambda)   \, \di m_\vf(\lambda) \\
\notag
& \geq \theta  \int\limits_{\lambda \in \sigma(H_0) } g_{I_0}(\lambda)^2 \,  \, \di m_\vf(\lambda) = \theta  \norm{g_{I_0}(H_0) \vf}_0^2  \ .
\end{align}
This  estimate and \eqref{mourre.har1} proves  that $H_0$ fulfills a Mourre estimate over  $I_0$. 
\end{proof}

To conclude this section, we recall that in  \cite{Mas19}  it is proved that the pseudodifferential operator 
\begin{align}
\label{V.har}
V(t) := e^{- \im t K_0} \, (S + S^*) \, e^{\im t K_0}  
\end{align}
is a universal transporter (see Definition \ref{def:transporter}). Using the abstract  Theorem \ref{thm:ab} we  prove its stability  under perturbations of class   $C^\infty(\T, \cA_0)$:
\begin{theorem}
\label{thm:har}
Consider equation \eqref{har.osc} with $V(t)$ defined in \eqref{V.har}.
There exist    $\epsilon_0, \tM >0$   such that $\forall W \in C^\infty(\T,\cA_0)$ with 
$\sup_t \wp^0_\tM(W(t)) \leq \epsilon_0$, 
the operator $V + \epsilon W$ is a transporter. 
 More precisely $\forall r >0$ there exist a solution $\psi(t) \in \cH^r$ of 
$
\im \pa_t \psi =  \big( \frac{- \pa_x^2 + x^2}{2}+  V(t)+ W(t) \big)\psi  $ 
 and constants $C, T >0$ such that
$$
\norm{\psi(t)}_r \geq C \la t \ra^r , \quad \forall t \geq T \ . 
$$
\end{theorem}
\begin{proof}
Again we verify Assumption III.  It is clear that  $V(t) \in C^\infty(\T, \cA_0)$ and that  $\la V \ra = S+S^*$, so it has the form \eqref{H0.har} with $k = 1$ and  $\fV_1 = 2$.
Then Lemma \ref{lem:har.mourre} implies that $\la V\ra$ fulfills a Mourre estimate.
\end{proof}

\subsection{Half-wave equation on $\T$}
The half-wave equation on $\T$ is   given by 
 \begin{equation}
 \label{halfwave}
 \im \pa_t  \psi = |D| \psi + V(t,x, D) \psi \ , \qquad x \in \T \ .
 \end{equation}
 Here  $|D| $ is the Fourier multiplier defined by 
$$
  |D| \psi := \sum_{j \in \Z  } |j|\,  \psi_j \, e^{\im j x}  \ , \qquad
  \psi_j := \frac{1}{2\pi} \int_\T \psi(x) e^{-\im j x} \di x \ ,
 $$
 whereas $V(t,x,D)$ is a pseudodifferential operator of order 0.
In this case  $K_0 := |D|+1$, the scale of  Hilbert spaces  defined  as 
 $\cH^r = {\rm Dom }\left(K_0^r\right)$ coincides with standard  Sobolev spaces on the torus $H^r(\T)$, and the base space $(\cH^0, \la \cdot, \cdot \ra)$ is $L^2(\T, \C)$   with its  standard  scalar product.
In this setting we shall use 
pseudodifferential operators with periodic symbols,  belonging to the following class:
\begin{definition}
A  function $ a(x,\xi)$ is a {\em periodic symbol of order} $\rho\in \R$ if $a \in C^\infty(\T_x \times \R_\xi, \C)$ and for any  $\alpha,  \beta \in \N_0 $, 
there exists a constant $C_{\alpha \beta} >0$ such that
\begin{equation}\label{simb-pro}
\abs{ \pa_x^\alpha \, \pa_\xi^\beta \, a(x, \xi) } \leq C_{\alpha \beta } \, \la \xi \ra^{\rho - \beta}  , \quad \forall x \in \T, \, \forall \xi \in \R \,  . 
\end{equation}
We will write $a \in S^\rho_{\per}$. We also put $ S^{- \infty}_\per := \bigcap_{\rho \in \R} S^{\rho}_\per$ the class of smoothing symbols.
\end{definition}

   We endow $S^\rho_\per$  with the family of seminorms
\begin{equation}
\label{semi}
\wp^\rho_j(a) := \sum_{|\alpha| + |\beta| \leq j}
\ \ \sup_{(x, \xi) \in \T \times \R} {\left|\partial_x^\alpha \, \partial_\xi^\beta a(x,\xi)\right| \,  \la \xi \ra^{-\rho+\beta} } \ , \qquad j \in \N_0 \ . 
\end{equation}
Such seminorms turn $S^\rho_\per$ into a Fr\'echet space.
If a symbol $a$ depends on additional parameters (e.g. it is time dependent), we ask that all the seminorms  are uniform w.r.t. such parameters.\\
 To a symbol $a \in S^\rho_\per$ we associate its quantization $a(x, D)$  acting on a $2\pi$-periodic function $u(x) = \sum_{j \in \Z} u_j e^{\im j x}$ as 
\begin{equation}
\label{opW} 
a(x, D)u := {\rm Op} {(a)}[u]  := \sum_{j \in \Z} \, a(x, j) \, 
   u_j \,  e^{\im j x } \,  . 
\end{equation}

\begin{remark}
\label{sc.per}
Given a symbol $a(\xi)$ independent of $x$, then $ {\rm Op} {(a)}$ is the Fourier multiplier operator 
$a(D) u = \sum_{j \in \Z} \, a(j) \, u_j \, e^{\im j x} $.
 If instead the symbol $a(x)$ is independent of $\xi$, then $ {\rm Op}{(a)}$ is the multiplication operator $ {\rm Op} {(a)} u = a(x) u $.
\end{remark}

\begin{definition}
We say that $A \in \cA_\rho$ if $A = {\rm Op} (a)$ with $a \in S^\rho_{\per}$.
\end{definition}

\begin{example}
The operator $|D| \in \cA_1$ with symbol given by 
$\td(\xi) := |\xi| \chi(\xi)$ where $\chi$ is an even, positive smooth cut-off function satisfying
$\chi(\xi) = 0$ for $|\xi| \leq \frac15$, $\chi(\xi) = 1$ for $|\xi| \geq \frac25$ and $\pa_\xi \chi(\xi) >0$ $\, \forall \xi \in (\frac15, \frac25)$.\\
Also the Fourier projectors $\Pi_\pm$ and $\Pi_0$ defined by
\begin{equation}
\label{fou.proj}
\Pi_+u  : = \sum_{j \geq 1} u_j \, e^{\im j x} , \qquad
\Pi_-u  : = \sum_{j \leq - 1} u_j \, e^{\im j x} , \qquad
\Pi_0 u  : =  u_0 
\end{equation}
are pseudodifferential operators. In particular $\Pi_\pm = \Op{\pi_\pm} \in \cA_0$ and $\Pi_0 = \Op{\pi_0} \in \cA_{- \infty}$, where  $\pi_\pm, \pi_0$ are a   smooth partition of unity, $\pi_+(\xi) + \pi_-(\xi) + \pi_0(\xi) = 1$ $\, \forall \xi$,  fulfilling 
\begin{equation}
\label{eta}
\pi_+(\xi) = 
\begin{cases}
1 & \mbox{ if } \xi \geq \frac45 \\
0 & \mbox{ if } \xi \leq \frac35
\end{cases} \ , 
\quad
\pi_-(\xi) = 
\begin{cases}
1 & \mbox{ if } \xi \leq -\frac45 \\
0 & \mbox{ if } \xi \geq -\frac35 
\end{cases}  \ , 
\quad
\pi_0(\xi)  = 
\begin{cases}
1 & \mbox{ if } |\xi| \leq \frac35 \\
0 & \mbox{ if } |\xi|  \geq \frac15 
\end{cases} \ . 
\end{equation}
\end{example}

%

In this setting we prove that any multiplication operator,  multiplied by an appropriate time periodic function, becomes a transporter. Here the result.

\begin{theorem}
\label{thm:per}
Let  $ v \in C^\infty(\T, \R)$. Choose $j \in \Z\setminus\{0\}$ such that the  Fourier coefficient $v_j \neq 0$. 
Then the selfadjoint operator
\begin{equation}
\label{V.per}
V(t,x):= \cos(jt) \, v(x)
\end{equation}
 is a transporter. More precisely, $\forall r >  0$  there exist  a
solution $\psi(t) \in \cH^r$ of 
$\im \pa_t \psi = \big( |D| + V(t,x) \big) \psi$ 
 and constants $C, T >0$ such that 
$$
\norm{\psi(t)}_r \geq C \la t \ra^r , \quad \forall t > T .
$$
\end{theorem}
The theorem follows from Theorem \ref{thm:ab0}. 
So first we put ourselves in the setting of the abstract theorem and  
 rewrite \eqref{halfwave} as
\begin{equation}
\label{vtilde.per}
\im \pa_t \psi = K_0 \psi + \wt V(t,x) \psi , \qquad \wt V(t,x):=  \cos(jt) v(x) - 1  \in \cA_0 \ .
\end{equation}
Again we check  Assumptions I-III. Regarding Assumption I, it is the usual pseudodifferential calculus for periodic symbols, see e.g. 
\cite{SarVai}.\\

\noindent{\bf Verification of Assumption II}.
One has   $\sigma(K_0) = \{ n \}_{n \in \N}$.  
To prove Assumption II $(ii)$ we use the identity 
$e^{-\im t  K_0} \fA \,  e^{ \im t K_0}  = e^{-\im t  |D|} \fA \,  e^{ \im t |D|}$ and Egorov theorem for $|D|$,  see e.g. \cite[Theorem 4.3.6]{sogge}.
Actually we need also the following version of Egorov theorem.

\begin{lemma}
\label{prop:egD}
Let  $a \in S^\rho_\per$, $\rho \in \R$. Then
\begin{equation}
\label{egD}
 e^{\im t  |D|} \, \Op{a} \,  e^{ -\im t |D|} = 
  \Op{a(x+ t, \xi) }\Pi_+  +   \Op{a(x- t, \xi) }\Pi_- +  R(t)
\end{equation}
where
$\Pi_\pm$ are defined in \eqref{fou.proj} and 
 $R(t) \in C^\infty(\T, \cA_{\rho-1})$.\\
 If $\Op{a}$ is selfadjoint, so is $ e^{\im t  |D|} \Op{a} \,  e^{ -\im t |D|}, $ $\, \forall t$.
\end{lemma}
\begin{proof}
The classical Egorov theorem for the half-Laplacian $|D|$  says that
$$
 e^{\im t  |D|} \, \Op{a} \,  e^{ -\im t |D|} = \Op{ a\circ \phi_\td^t(x, \xi)} + \wt R(t)
 $$
 where $\phi_\td^t(x, \xi)$ is the time $t$ flow of the classical Hamiltonian $\td(\xi) = |\xi| \chi(\xi)$ (the symbol of $|D|$) and $\wt R(t) \in C^\infty(\R, \cA_{\rho-1})$, see e.g.  \cite[Theorem 4.3.6]{sogge}.
 \\
We  compute more explicitly $a\circ \phi_\td^t(x, \xi)$. The Hamiltonian equations of $\td(\xi)$ and its flow  $\phi^t_\td$ are given by 
$$
\begin{cases}
\dot x = \pa_\xi \td(\xi) = \td'(\xi) \\
\dot \xi = - \pa_x \td(\xi) = 0
\end{cases} \ , 
\qquad \phi^t_\td(x, \xi) = (x + t \td'(\xi), \, \xi) \ .
$$
As $\td'(\xi) = 1$ for $\xi \geq \frac25$ and $\td'(\xi) = -1$ for $\xi \leq - \frac25$, we write
$$
(a\circ \phi_\td^t)(x, \xi) = 
a(x + t, \xi) \, \pi_+(\xi) + 
a(x - t, \xi) \, \pi_-(\xi) + 
a(x  + t \td'(\xi), \xi)\,  \pi_0(\xi)  \ . 
$$
As $\pi_0 \in S^{-\infty}_\per$, the operator $\Op{a(x+t\td'(\xi), \xi) \, \pi_0(\xi)} \in C^\infty(\R, \cA_{- \infty})$. 
Moreover by symbolic calculus  
$$
\Op{a(x\pm t, \xi) \, \pi_\pm(\xi)} =  \Op{a(x\pm t, \xi) }\Pi_\pm + R_\pm(t) , \quad R_\pm(t) \in C^\infty(\R, \cA_{\rho-1}) \ . 
$$
Formula \eqref{egD} follows with
$
R(t):= \wt R(t) + R_+(t) + R_-(t) +\Op{a(x+t\td'(\xi), \xi) \, \pi_0(\xi)} 
$.
We claim that  $R(t)$   is periodic in time. This follows by difference since
both  $e^{\im t  |D|} \, \Op{a} \,  e^{ -\im t |D|}$ and   
$  \Op{a(x\pm t, \xi) }\Pi_\pm$ are periodic in $t$ (recall that the symbol $a(x, \xi)$ is periodic in $x$).\\ 
Finally  as $e^{\pm \im  t |D|}$ are unitary, the claim on the selfadjointness of $e^{ \im t |D|} \, \Op{a} \, e^{-\im t |D|}$ follows.
\end{proof}

\noindent{\bf Verification of Assumption III}.
First we compute $\langle \wt V \rangle$. 

\begin{lemma}
\label{hw.HO2}
The resonant average $\langle \wt V \rangle \in \cA_0$ of  $\wt V$ (defined in \eqref{vtilde.per}) is given by
\begin{equation}
\label{H0.per}
\langle \wt V \rangle = \tv(x) - 1  + R , \qquad \tv(x) := \Re (v_j e^{\im  j x})
\end{equation}
and   $R \in \cA_{-1}$ is selfadjoint.
\end{lemma}
\begin{proof}
First remark that, as $e^{\im t K_0} = e^{\im t |D|} e^{\im t}$,
\begin{align}
\label{H0.per1}
\langle \wt V \rangle
& =  \frac{1}{2\pi} \int_0^{2\pi} e^{\im s |D|} \, \wt V(s) \, e^{- \im s |D|} \di s  
= 
\frac{1}{2\pi} \int_0^{2\pi} \cos(js) \, e^{\im s |D|} \, v(x) \, e^{- \im s |D|} \, \di s -  1  \ . 
\end{align}
We  compute $e^{\im s |D|} \, v(x) \, e^{- \im s |D|}$ with the aid of  Lemma  \ref{prop:egD}, getting 
\begin{equation}
\label{eg.per}
e^{\im s |D|} \, v(x) \, e^{- \im s |D|} = 
v(x+ s) \, \Pi_+  + 
v(x- s) \, \Pi_-  +  \wt R(s) , 
\end{equation}
where $\wt R(s) \in C^\infty(\T, \cA_{-1})$. Then, recalling that $v_j = \bar{v_{-j}}$ being $v(x)$ real valued,
\begin{align*}
 \frac{1}{2\pi}\int_0^{2\pi}
\cos(j s) \, e^{\im s |D|} \, v(x) \, e^{- \im s |D|} \, \di s  
&\stackrel{\eqref{eg.per}}{ =} \sum_{\sigma = \pm}
\frac{1}{2\pi}\int_0^{2\pi} \cos(js) \, v(x\sigma   s) \,\di s  \, \Pi_\sigma  +
\frac{1}{2\pi}\int_0^{2\pi} \cos(js) \wt R(s)  \di s \\
& = \Re \left( v_j e^{\im j x}\right)  \, ( \Pi_+ +   \Pi_- )  + \frac{1}{2\pi}\int_0^{2\pi} \cos(js) \wt R(s)  \di s\\
& =   \Re \left( v_j e^{\im j x}\right) + R
\end{align*}
where
$R := \frac{1}{2\pi}\int_0^{2\pi} \cos(js) \wt R(s)  \di s - \Re \left( v_j e^{\im j x}\right)\Pi_0  \in \cA_{-1}$.
Together with \eqref{H0.per1}, this proves  \eqref{H0.per}. 
Finally  $R$ is selfadjoint by difference, since both $\langle \wt V \rangle$ and $\tv(x)-1$ are selfadjoint operators.
\end{proof}

 Define the selfadjoint operator 
\begin{align}
\label{A.per}
A :=   \tw(x) \, \frac{\pa_x}{\im} + \frac{\pa_x}{\im} \, \tw(x)  \ , 
\qquad
\tw(x) := \Im( v_j e^{\im j x} )  \ 
\end{align}
belonging  to $\cA_1$.
The next lemma verifies Assumption III.
\begin{lemma}
Assume that $v_j \neq 0$.   The following holds true:
\begin{itemize}
\item[(i)] The operator $H_0:= \langle  \wt V \rangle$ has spectrum $\sigma(H_0) \supseteq [-|v_j| - 1, \, |v_j| - 1] =: I $.
\item[(ii)] 
For any interval $I_0 \subset  I $, any $g_{I_0} \in C^\infty_c(\R, \R_{\geq 0})$ with $g_{I_0} \equiv 1$ over $I_0$
and ${\rm supp }\, g_{I_0} \subset   I $, there exist $\theta >0$ and a compact operator $\fK$  such that 
$$
g_{I_0}(H_0) \, \im [H_0, A] \, g_{I_0}(H_0) \geq \theta \,  g_{I_0}(H_0)^2  + \fK \ .
$$
Here  $A$ is defined in \eqref{A.per}.
\end{itemize}
\end{lemma}
\begin{proof}
During the proof we shall use that any operator in $\cA_{-1}$ is compact.   Moreover we shall simply denote  any compact operator
by $\fK$, which can change from line to line.\\
$(i)$ By Lemma \ref{hw.HO2},  $H_0$ is a compact perturbation of the multiplication operator by $\tv(x) -1$, whose spectrum coincides with $ I$.   
Then by Weyl's theorem 
$$\sigma(H_0) \supseteq \sigma_{ess}(H_0) = \sigma_{ess}(\tv(x) -1) =  I  \ . $$
$(ii)$ 
First notice that, as  $\tv(x) = \Re(v_j e^{\im j x})$ and $\tw(x) = \Im(v_j e^{\im j x})$, one has the identities 
\begin{equation}
\label{vw}
\tv(x)^2 + \tw(x)^2 = |v_j|^2 , \qquad \tv'(x) = - j \, \tw(x) \ . 
\end{equation}
Next we compute 
\begin{align*}
\im [H_0, A] 
& = \im [\tv(x) - 1 + R, A] 
= - 2\tw(x) \, \tv'(x) + \im [R, A]  \\
& \stackrel{\eqref{vw}}{=} 2j  \big(|v_j|^2  - \tv(x)^2\big)  +  \fK  = 2j \big(|v_j|^2  - (H_0 + 1- R)^2\big)  +  \fK  \\
& = 2j \big(|v_j|^2  - (H_0 +1)^2 \big)  + \fK  \ . 
\end{align*}
Putting $f(\lambda):=  2j \big( |v_j|^2 -( \lambda+1)^2 \big)  $, we get  $\forall \vf \in \cH$
\begin{align}
\label{M.per}
\la g_{I_0}(H_0) \, \im [H_0, A]   \, g_{I_0}(H_0) \vf, \vf \ra 
 =  \la g_{I_0}(H_0) \, f(H_0) \,  g_{I_0}(H_0) \vf, \vf \ra  + 
\la \fK \vf, \vf \ra .
\end{align}
Now we notice that  $f(\lambda) $ is positive in the interior of $ I$; 
so we put
$$
\theta :=  \inf\left\lbrace f(\lambda) \colon \ \ \lambda \in {\rm supp }\, g_{I_0} \right\rbrace > 0 \ . 
$$
With this information we apply the spectral theorem, getting, as in the previous section, 
\begin{align}
\la g_{I_0}(H_0) \, f(H_0) \,  g_{I_0}(H_0) \vf, \vf \ra
& \geq \theta  \int\limits_{\lambda \in  I } g_{I_0}(\lambda)^2 \,  \, \di m_\vf(\lambda) = \theta  \norm{g_{I_0}(H_0) \vf}_0^2   \ . 
\end{align}
This together with  \eqref{M.per}  establishes the  Mourre estimate over $I_0$. 
\end{proof}

	\appendix

	\section{Flows of pseudodifferential operators}
	\label{app:flow}
In this appendix we collect some known results about the flow generated by pseudodifferential operators belonging to the  algebra  $\cA$. 
The setting is the same as \cite{BGMR2} and we refer to that paper for the proofs. 
The first result describes how a Schr\"odinger equation is changed under a change of variables induced by the flow of a pseudodifferential operator, see Lemma 3.1 of \cite{BGMR2}:
\begin{lemma}
\label{T.1}
Let $H(t)$ be a time dependent selfadjoint operator, and $X(t)$ be a selfadjoint family of operators.  Assume that
$\psi(t)=e^{-\im X(t)}\vf(t)$ then
\begin{equation}
\label{1}
\im\pa_t \psi=H(t) \psi\ \quad \iff \quad \im\pa_t \vf  =  H^+(t) \vf 
\end{equation}  
where 
\begin{align}
\label{4.1.1}
 H^+(t):= e^{\im X(t)} \, H(t)\, e^{-\im X(t)}
-\int_0^1 e^{\im s X(t) } \,(\pa_t X(t))  \, e^{-\im
  s X(t) }  \ \di s \ .
\end{align}
\end{lemma}
The next property we shall need is the Lie expansion of $e^{\im
  X} \, A\, e^{-\im X}$ in operators of decreasing order, see  Lemma 3.2 of \cite{BGMR2}:
\begin{lemma}
\label{comX}
Let $X\in\cA_{\rho}$ with $\rho< 1$ be a symmetric operator. Let $A\in\cA_{m}$ with
$m \in \R$. Then $e^{\im
\tau   X} \, A\, e^{-\im \tau X}$ is selfadjoint and for any $M\geq 1$ we have\footnote{in \cite{BGMR2} we have defined $ {\rm ad}_X(A) = \im [X, A]$ rather than  \eqref{adjoint}; so we formulate the next result with the current notation }
\begin{equation}
\label{expansion}
e^{\im \tau X}\, A \,  e^{-\im\tau X } = \sum_{\ell=0}^{M}\frac{\tau^{\ell}}{\im^\ell \, \ell!}{\rm ad}_X^\ell(A) + R_M(\tau, X,A) \ , \qquad \forall \tau \in \R \ , 
\end{equation}
 where  $R_M(\tau, X,A)\in\cA_{m -(M+1)(1-\rho)}$.  
\\
In particular ${\rm ad}_X^{\ell}(A)\in\cA_{m-\ell(1-\rho)}$ and  
$e^{\im \tau X}\, A \, e^{-\im \tau X}\in \cA_{m}$, $\forall \tau
\in \R$.
\end{lemma}
The last result concerns  boundedness properties of the operator $e^{- \im \tau X}$, see  Lemma 3.3 of \cite{BGMR2}:
\begin{lemma}
\label{MR}
 Assume that $X(t)$ is a family of selfadjoint   operators in $\cA_1$ s.t. 
 \begin{equation}
 \label{X(t).est}
 \sup_{t \in \R} \wp^1_j(X(t)) < \infty \ , \quad 
\forall j \geq 1 \ . 
 \end{equation}
 Then $e^{- \im \tau X(t)}$ extends to an operator in $\cL(\cH^r)$   $\, \forall r \in \R$, and moreover there exist $c_r, C_r >0$ s.t.
 \begin{equation}
 \label{X(t).est0}
 c_r \norm{\psi}_r \leq \norm{e^{-\im \tau X(t)} \psi}_r \leq 
 C_r \norm{\psi}_r \ , \qquad \forall t \in \R \ , \quad \forall \tau \in [0,1]  \ .
 \end{equation}
\end{lemma}

	\section{Functional calculus }
	\label{app:fc}

In this section we  collect some known results about functional calculus of selfadjoint operators  which are used thorough the paper. 
We begin recalling  	Helffer-Sj\"ostrand formula \cite{HS}, following  the presentation of \cite{davies}.

\begin{definition}
A function $f \in C^\infty(\R, \C)$ will be said to belong to the class $S^\rho$,  $\rho \in \R$,  if $\forall m \in \N_0$, $\exists C_m >0$ such that 
$$
\abs{\frac{\di^m }{\di x^m}f(x) } \leq C_m \la x \ra^{\rho  - m} , \quad \forall x \in \R \ . 
$$
\end{definition}	
As usual we  set the seminorms
	$$
	\wp^\rho_m(f) := \sum_{0 \leq j \leq m} \ \sup_{x \in \R} \abs{\frac{\di^m f(x)}{\di x^m} } \, \la x \ra^{-\rho  + m} \ , \qquad m \in \N_0 \ . 
	$$
Given $f \in S^\rho$, we define its {\em almost analytic extension} as follows: for any  $N \in \N$, put	
	\begin{equation}
	\label{}
\wt f_N \colon \R^2 \to \C , \qquad 	\wt f_N(x, y)  :=  \left( \sum_{\ell = 0}^N f^{(\ell)}(x) \frac{(\im y)^\ell}{\ell !} \right) \, \tau \left(\frac{y}{\la x \ra} \right)
	\end{equation}
	where $\tau \in C^\infty(\R, \R_{\geq 0})$ is a  cut-off function fulfilling $\tau(s) = 1$ for $|s| \leq 1$ and $\tau(s) = 0$ for $|s| \geq 2$.
	It is well known \cite{davies} that the choice of $N$ and of the cut-off function $\tau$  are by no means critical, and even other choices of $\widetilde f_N$ are possible (see e.g. \cite{DG}).
The following properties  are true \cite{davies}:   let $f \in S^\rho$ with $\rho <0$,  then
\begin{align}
& \wt f_N \vert_\R = f  , \qquad 
 {\rm supp }\, \wt f_N \subset \left\lbrace x + \im y \ \colon  \  \ \
x \in {\rm supp }\,  f 
, \quad 
|y| \leq 2 \la x \ra 
\right \rbrace \ , \\
& \abs{\frac{\pa \wt f_N(x,y)}{\pa {\bar z}} } \leq C_N \la x \ra^{\rho - N - 1} \, |y|^N , \qquad
\frac{\pa \wt f_N}{\pa {\bar z}} := \left( \frac{\pa \wt f_N}{\pa x} + \im \frac{\pa \wt f_N}{\pa y} \right) \\
\label{est.aae2}
& \int_{\R^2} \abs{\frac{\pa \wt f_N (z)}{\pa {\bar z}} } \, 
\abs{{\rm Im }\,(z)}^{-p-1}  \di \bar z  \wedge \di z  \leq C_N \, \wp^\rho_{N+2}(f) , \qquad \forall p = 0, \ldots, N ,
\end{align}
where $z  = x + \im y$ and  $\di \bar z  \wedge \di z $ is the Lebesgue measure on $\C$.

\noindent Given 	$\fH$ a selfadjoint operator and $f \in S^\rho$, $\rho <0$,  the {\em Helffer-Sj\"ostrand formula} defines $f(\fH)$ as 
\begin{equation}
	\label{f(T)}
	f(\fH)  := \frac{\im}{2\pi} \int_{\R^2} \frac{\pa \wt f_N(z)}{\pa {\bar z}} \, (z-\fH)^{-1} \di \bar z  \wedge \di z  = - \frac{1}{\pi} \int_{\R^2} \frac{\pa \wt f_N(z)}{\pa {\bar z}} \, (z-\fH)^{-1} \di x \, \di y  \ . 
	\end{equation}	
%
\begin{theorem}[\cite{davies}]
\label{thm:davies}
Let $f \in S^\rho$, $g \in S^\mu$ with $\rho, \mu <0$ and $\fH$  a selfadjoint operator. Then 
\begin{itemize}
\item[(i)] The operator $f(\fH)$ is independent of $N$ and of the cut-off function $\tau$.
\item[(ii)] The integral in  \eqref{f(T)} is norm convergent 
and $ \norm{f(\fH)}_{\cL(\cH)} \leq \norm{f}_{L^\infty}$.
\item[(iii)]  $f(\fH) \,g(\fH) = (fg)(\fH)$.
\item[(iv)] $\bar f(\fH) = f(\fH)^*$.
\item[(v)] If $f \in C^\infty_c$ has support disjoint from $\sigma(\fH)$, then $f(\fH) = 0$.
\item[(vi)] If $z \notin \R$ and $f_z(x):= (z- x)^{-1}$ for all $x \in \R$, then $f_z \in S^{-1}$ and $f_z(\fH) = (z-\fH)^{-1}$. 
\end{itemize}
\end{theorem}
\begin{remark}
\label{rem:st}
Given $f \in S^\rho$, $\rho <0$  and $\fH$ selfadjoint, the operator $f(\fH)$ defined via Helffer-Sj\"ostrand formula coincides with the classical definition given by the spectral theorem, namely
$$
f(\fH) = \int_{\R} f(\lambda) \, \di E(\lambda)
$$
where $\di E(\lambda)$ is the spectral resolution of $\fH$.
For a proof, see e.g. \cite{DiSj}, Theorem 8.1.
\end{remark}
 Next we recall  expansion formulas for commutators.
We start from the basic identities 
	\begin{align}
		\label{ad2}
	{\rm ad }^n_\fA(\fP \fQ) & = \sum_{k=0}^n \binom{n}{k} \, {\rm ad }^{n-k}_\fA(\fP) \, {\rm ad }^k_\fA(\fQ) , \qquad 
	 [\fP, \fA^n]  =  \sum_{j = 1}^n c_{n,j} \, {\rm ad }^j_\fA(\fP) \, \fA^{n-j} \ . 
	\end{align}
	For the next lemma see e.g. \cite[Lemma C.3.1]{DG} or \cite[Appendix B]{HuSi}.
	\begin{lemma}[Commutator expansion  formula]
	\label{lem:DG}
	Let  $k \in \N $ and  
	 $\fA, \fB$   selfadjoint operators with 
	$$
	\norm{{\rm ad}_\fA^j(\fB)}_{\cL(\cH)} < \infty , \qquad  \forall \, 1 \leq j \leq k \ . 
	$$
	Let  $f \in S^\rho$ with $\rho < 0$, then one has the  right and left   commutator expansions
	\begin{align}
	\label{r.exp}
		[\fB, f(\fA)] &  = \sum_{j = 1}^{k-1} \frac{1}{j!} \, f^{(j)}(\fA) \, {\rm ad}_\fA^j(\fB) + R_k(f, \fA, \fB) \\
	\label{l.exp}
		& = \sum_{j = 1}^{k-1} \frac{(-1)^{j-1}}{j!} \, {\rm ad}_\fA^j(\fB) \,  f^{(j)}(\fA) \,+ \wt R_k(f, \fA, \fB)
	\end{align}
	where the operators  $R_k, \wt R_k$ fulfill 
	\begin{equation}
	\label{R.exp}
	\norm{R_k(f, \fA, \fB)}_{\cL(\cH)} , \quad
	\norm{\wt R_k(f, \fA, \fB)}_{\cL(\cH)} \leq  \, C_N \, \wp^\rho_{k+2}(f)  \, \norm{{\rm ad}_\fA^k(\fB)}_{\cL(\cH)}  \ . 
	\end{equation}
	\end{lemma}

\begin{lemma}
\label{ad.A.g(H)}
Let $k \in \N$ and 
	 $\fA, \fH$  selfadjoint operators such that 
	\begin{equation}
	\label{adAH}
\norm{{\rm ad}_\fA^j(\fH)}_{\cL(\cH)} < \infty , \qquad \forall \,  1 \leq j \leq k \ . 
	\end{equation}
	Let $g \in S^\rho$ with $\rho <0$. Then
	\begin{equation*}
\norm{	{\rm ad}_\fA^j( g(\fH)) }_{\cL(\cH)}  < \infty \, \quad \forall 1 \leq j \leq k   \ .
	\end{equation*}
\end{lemma}
\begin{proof}
Take $N \geq k$ and use 
 Helffer-Sj\"ostrand formula to write 
\begin{equation}
\label{ad.gJ.K0}
{\rm ad }^j_{\fA}(g(\fH)) = \frac{\im}{2\pi} \int_{\R^2} \frac{\pa \wt g_N(z)}{\pa \bar z} \, {\rm ad }^j_{\fA}\big((z - \fH)^{-1}\big) \, \di \bar z \wedge \di z  . 
\end{equation}
As
${\rm ad }_{\fA}\big((z - \fH)^{-1}\big) = (z - \fH)^{-1} \,  {\rm ad }_{\fA}(\fH)  \, (z - \fH)^{-1}$,   by induction one gets for $j = 1, \ldots, k$
$$
{\rm ad }_{\fA}^j\big((z - \fH)^{-1}\big)
= \sum_{\ell = 1}^j \sum_{k_1 + \cdots +  k_\ell = j \atop k_1, \ldots , k_\ell \geq 1} 
c_{k_1 \cdots k_\ell}^{\ell,j} \, 
(z - \fH)^{-1} \,  {\rm ad }_{\fA}^{k_1}(\fH)  \, (z - \fH)^{-1} 
{\rm ad }_{\fA}^{k_2}(\fH)
\cdots
(z - \fH)^{-1} \,  {\rm ad }_{\fA}^{k_\ell}(\fH)  \, (z - \fH)^{-1}
$$
Using  \eqref{adAH} and the estimate $
\norm{ (z - \fH)^{-1}}_{\cL(\cH)} \leq {\abs{{\rm Im }\, (z)}^{-1}}$, $  \forall z \in \C\setminus \R$,  one has 
for $j = 1, \ldots, k$
\begin{equation*}
\norm{{\rm ad }_{\fA}^j\big((z - \fH)^{-1}\big)}_{\cL(\cH)}
\leq 
\sum_{\ell = 1}^j C_\ell \,  {\abs{{\rm Im }\, (z)}}^{-\ell -1} , \quad \forall z \in \C\setminus \R . 
\end{equation*}
 Inserting this estimate into
 \eqref{ad.gJ.K0} and using  \eqref{est.aae2} we bound
 for any $j = 1, \ldots, k$ 
 $$
\norm{ {\rm ad }^j_{\fA}(g_J(\fH)) }_{\cL(\cH)}
\lesssim 
  \sum_{\ell = 1}^j 
 \int_{\R^2} 
\abs{\frac{\pa \wt g_N(z)}{\pa \bar z}}  \, \,  {\abs{{\rm Im }\, (z)}}^{-\ell -1} \di \bar z \wedge \di z 
\lesssim \wp^\rho_{N+2}(g) <  \infty \ .$$
\end{proof}

\begin{lemma}
\label{lem:g}
 Let $g \in C^\infty_c(\R, \R)$. Let  $\fH, \fB \in \cL(\cH)$ be  selfadjoint. Then $\exists \, C >0$ such that 
\begin{equation}
\label{lem:g.est}
\norm{g(\fH + \fB ) - g(\fH)}_{\cL(\cH)} \leq C \norm{\fB}_{\cL(\cH)} \ . 
\end{equation}
If $\fB$ is compact on $\cH$, so is $g(\fH+ \fB ) - g(\fH)$.
\end{lemma}
\begin{proof}
Take $N \geq 1$. Using Helffer-Sj\"ostrand formula and the resolvent identity we obtain
$$
g(\fH + \fB) - g(\fH) = 
 \frac{\im}{2\pi} \int_{\R^2} 
 \frac{\pa \wt g_N(z)}{\pa \bar z} \, \big(z - (\fH+ \fB)\big)^{-1}\, \fB \, (z-\fH)^{-1} \, \di \bar z \wedge \di z  \  . 
$$
Then use  $\norm{\big(z - (\fH+ \fB)\big)^{-1}}_{\cL(\cH)} $, 
$\norm{(z - \fH)^{-1}}_{\cL(\cH)} \leq \abs{{\rm Im} (z)}^{-1}$ for $z \in \C \setminus \R$ and
 \eqref{est.aae2}.\\
 If $\fB$ is compact then $\big(z - (\fH+ \fB)\big)^{-1}\, \fB \, (z-\fH)^{-1} $ is a compact operator for any $z \in \C \setminus  \R$.
\end{proof}

\section{Local energy decay estimates}
\label{app:minimal}
In this section we prove a local energy decay estimate starting from Mourre estimate. The result is essentially known   but  we could not find in the literature a statement exactly as the one we use in the paper,  so we  include here a  proof, which follows closely the one of Lemma 4.1 of   
\cite{GNRS}.
In  this part we do not require  pseudodifferential properties of the operators. We shall assume conditions (M1) and (M2) at page \pageref{M1}.
\begin{theorem}[Local energy decay estimate]
\label{thm:SS}
Fix $k \in \N$ and 
assume {\rm (M1)--(M2)}  with $\tN \geq 4k+2$ and $\fK = 0$. 
Then for any interval $J \subset I$, any function $g_J \in C^\infty_c(\R, \R_{\geq 0})$ with ${\rm supp }\, g_J \subset I$, $g_J = 1$ on $J$,  there exists $C >0$ such that 
\begin{equation}
\label{SS}
\norm{ \la \fA \ra^{-k} \, e^{- \im \fH t} \, g_J(\fH) \, \psi}_0 \leq C \la t \ra^{-k} \norm{ \la \fA \ra^{k}\, g_J(\fH) \psi}_0  , \quad \forall t \in \R , 
\end{equation}
for any $\psi$ such that the r.h.s. is finite.
\end{theorem}

\begin{proof} 
Take $\chi(\xi) := \frac12(1- \tanh \xi)$. Put $\eta(\xi):= \frac{1}{\sqrt{2} \cosh \xi}$ and note that
\begin{equation}
 \label{chieta}
 \chi' = - \eta^2 , \qquad \abs{\eta^{(m)}(\xi)} \leq C_m \, \eta(\xi) , \quad \forall \xi \in \R, \ \ \forall m \in \N . 
 \end{equation} 
Next we set for $a \in \R$, $s \geq 1$ and $\vartheta := \frac{\theta}{2}$ (with $\theta$ of (M2) )
$$
\fA_{t,s} := \frac{1}{s}\big( \fA - a - \vartheta t\big) 
$$
and define via functional calculus the operators $\chi(\fA_{t,s})$ and $\eta(\fA_{t,s})$; both  are  bounded and  selfadjoint on $\cH$. To shorten the notation, from now on  we  write $\chi_{t,s} \equiv \chi(\fA_{t,s})$, $\eta_{t,s} \equiv \eta(\fA_{t,s})$,   $g_J \equiv g_J(\fH)$ and $\psi_t := e^{- \im t \fH} \psi$.
 Note that $e^{- \im \fH t} g_J(\fH)   \psi =
 g_J(\fH)     e^{- \im \fH t} \psi \equiv  g_J \psi_t$. \\
 The starting point of the proof is an energy estimate for the quantity $\norm{(\chi_{t,s})^{\frac12}  \, g_J \psi_t}_0$. We have
\begin{equation}
\label{min1}
\frac{d}{dt} \norm{(\chi_{t,s})^{\frac12}  \, g_J \psi_t}_0^2 
= 
\frac{d}{dt}\la \chi_{t,s} g_J \psi_t, \, g_J \psi_t \ra 
=
\frac{\vartheta}{s} \norm{\eta_{t,s} \, g_J \,  \psi_t}_0^2 + \la \im [\fH, \chi_{t,s}] 	\,  g_J \psi_t, \,  g_J \psi_t \ra . 
\end{equation}
To evaluate the right hand side we shall use the   commutator formulas  in Lemma \ref{lem:DG},  
 the identity
\begin{equation}
\label{ad.Ats}
{\rm ad}^j_{ \fA_{t,s}}(\fH) = \frac{1}{s^j} {\rm ad}^j_{\fA}(\fH) \ , \quad \forall s \geq 1 , \ \ \forall 1 \leq  j \leq \tN 
\end{equation}
and the fact that  all the operators ${\rm ad}^j_{\fA}(\fH)$ are bounded $\forall 1 \leq  j \leq \tN$ by (M1). 
The goal now is to estimate the second term in the right hand side of \eqref{min1}. 
For an arbitrary $f \in \cH$  we write 
\begin{equation}
\label{min2}
\la \im [\fH, \chi_{t,s}] f, f \ra  \stackrel{\eqref{r.exp}, \eqref{ad.Ats}}{=}  - \frac{1}{s} \la \eta_{t,s}^2 \im [\fH, \fA] f, f \ra
+
\sum_{j= 2}^{\tN-1}
\frac{1}{j!} \frac{1}{s^j}
 \la \chi_{t,s}^{(j)}\, \im \, {\rm ad}_\fA^j (\fH) f , f \ra  + \frac{1}{s^\tN} \la R_\tN f, f \ra
\end{equation}
where  $\chi_{t,s}^{(j)} := \chi^{(j)}(\fA_{t,s})$ and the remainder $R_\tN$ fulfills the estimate (see \eqref{R.exp})
\begin{equation}
\label{min3}
\norm{R_\tN }_{\cL(\cH)} \leq C_\tN \, \norm{ {\rm ad}_\fA^\tN (\fH)} \leq C_{\tN} . 
\end{equation}
Note that the constant $C_\tN>0$ in the previous estimate is uniform in $a \in \R$. In the following we shall simply denote by $R_\tN$ any bounded operator fulfilling an estimate like \eqref{min3}.\\
Consider now the first term in the expansion \eqref{min2} above. 
This time we use the left expansion 
\eqref{l.exp} and write
\begin{align}
\label{min30}
\frac{1}{s} \la \eta_{t,s}^2\,  \im [\fH, \fA] f, f \ra
& = 
\frac{1}{s}  \la \,  \im [\fH, \fA] \, \eta_{t,s} f,  \, \eta_{t,s} f \ra 
\\
& \quad +
\sum_{j=2}^{\tN-1}
\frac{(-1)^{j-1}}{(j-1)!}
\frac{1}{s^{j}} 
\langle  \, \im \, {\rm ad}_\fA^{j}(\fH) \, \eta_{t,s}^{(j-1)}f , 
\eta_{t,s} f
\rangle
+ 
\frac{1}{s^\tN}
\la  R_\tN f , f \ra 
\end{align}
where $R_\tN $ is estimated as in \eqref{min3}.
Consider now the second term in \eqref{min2}. 
From $\chi' = - \eta^2$, we have by functional calculus 
$
\chi^{(j)}(\fA_{t,s})  = \sum_{\ell = 1}^j c_{\ell j} \,  \eta^{(j-\ell)}(\fA_{t,s}) \, \eta^{(\ell)}(\fA_{t,s})  \, .
$
Thus we get that  
\begin{align}
\label{min31}
\frac{1}{s^j} \langle \chi_{t,s}^{(j)}\, \im \, {\rm ad}_\fA^j (\fH) f , f \rangle
&\stackrel{  \eqref{l.exp}}{=} 
\frac{1}{s^j} \sum_{\ell = 1}^j c_{\ell j} \langle  \, \im \, {\rm ad}_\fA^j (\fH) \, \eta_{t,s}^{(\ell)}f , \eta_{t,s}^{(j-\ell)}  f \rangle 
 \\
 & \quad + 
\sum_{\ell = 1}^j\sum_{n=1}^{\tN-j-1}
 \frac{c_{\ell jn}}{s^{j+n}}
  \langle  \,  
\im   \, {\rm ad}_\fA^{j+n} (\fH) \,  \eta_{t,s}^{(\ell+n)} f, \,  \eta_{t,s}^{(j-\ell)}  f \rangle  + \frac{1}{s^{\tN}} \la R_\tN f , f \ra 
\end{align}
By \eqref{min2}, \eqref{min30}, \eqref{min31} we have found that $\la \im [\fH, \chi_{t,s}] f, f \ra$ is a sum of terms of the form 
\begin{align*}
\la \im [\fH, \chi_{t,s}] f, f \ra = 
- \frac{1}{s} 
 \la \,  \im [\fH, \fA] \, \eta_{t,s} f,  \, \eta_{t,s} f \ra 
+
\sum_{j=2}^{\tN-1}
\frac{1}{s^{j}} 
\sum_{n, \ell,  m}
\langle  \, R_n \, \eta_{t,s}^{(\ell)}f , 
\eta_{t,s}^{(m)} f
\rangle
+ 
\frac{1}{s^\tN}
\la R_\tN f , f \ra 
\end{align*}
where $R_n, R_\tN$ are  bounded operators.
Furthermore, from the second of \eqref{chieta} and the spectral theorem, we bound 
\begin{equation}
\label{min4}
\abs{\langle  \, R_n \, \eta_{t,s}^{(\ell)}f , 
\eta_{t,s}^{(m)} f \rangle }
\leq
C \, \norm{\eta_{t,s} f}_0^2 \ . 
\end{equation}
We thus obtain,  for any $f \in \cH$ and $s \geq 1$, the estimate
\begin{equation}
\label{min5}
\la \im [\fH, \chi_{t,s}] f, f \ra  \leq - \frac{1}{s} \la \,  \im [\fH, \fA] \, \eta_{t,s} f,  \, \eta_{t,s} f \ra 
 + \frac{C_\tN}{s^2}  \norm{\eta_{t,s} f}_0^2 
 + \frac{C_\tN}{s^\tN}  \norm{f}_0^2 \ . 
\end{equation}
Now we evaluate such inequality at  $f = g_J \psi_t$, getting
\begin{equation}
\label{min50}
\la \im [\fH, \chi_{t,s}] g_J \psi_t, \, g_J \psi_t \ra  \leq - \frac{1}{s} \la \,  \im [\fH, \fA] \, \eta_{t,s} g_J \psi_t,  \, \eta_{t,s} g_J \psi_t \ra 
 + \frac{C_\tN}{s^2}  \norm{\eta_{t,s} g_J \psi_t}_0^2 
 + \frac{C_\tN}{s^\tN}  \norm{g_J \psi_t}_0^2 \ . 
\end{equation}
The next step is to prove that the first term in the right hand side above has a sign, up to higher order terms in $s^{-j}$. This is the point where the Mourre estimate (M2) comes into play.
 To see this, we analyze 
\begin{align}
\label{HAeta}
 \la \,  \im [\fH, \fA] \, \eta_{t,s} \, g_J \psi_t ,  \, \eta_{t,s}\,  g_J \psi_t \ra 
\equiv 
 \la \,  \im [\fH, \fA] \, \eta_{t,s} \, g_I \,  g_J \psi_t ,  \, \eta_{t,s} g_I \,  g_J \psi_t \ra  
\end{align}
where we used that $g_J g_I = g_J$.
Next we commute and expand in commutators  $\eta_{t,s} g_I$:
\begin{align}
\label{eta.gI}
\eta_{t,s} g_I = 
g_I \eta_{t,s} + [\eta_{t,s},  g_I ]  
\stackrel{ \eqref{l.exp} }{=} 
g_I \, \eta_{t,s} + \sum_{j = 1}^{\tN-2} \frac{c_j}{s^j}
\,  {\rm ad}_\fA^j (g_I(\fH)) \, \eta_{t,s}^{(j)}  + \frac{1}{s^{\tN-1}} \wt R_{\tN  -1} \ ; 
\end{align}
note that  Lemma \ref{ad.A.g(H)} assures that
the operators $ {\rm ad}_\fA^j (g_I(\fH)) $ are bounded $\forall j = 1, \ldots, \tN$, so is the operator $\wt R_{\tN-1}$ which fulfills 
\begin{equation}
\label{min6}
\norm{\wt R_{\tN-1}}_{\cL(\cH)} \leq C_\tN \, {\rm ad}_\fA^{\tN-1} (g_I(\fH)) < \infty  \ . 
\end{equation}
Again in the following we shall denote by $\wt R_{\tN - 1}$ any operator fulfilling an estimate like \eqref{min6}.
 Inserting the expansion \eqref{eta.gI} into \eqref{HAeta} 
 one gets, with $w := g_J \psi_t$,
 \begin{align*}
 \eqref{HAeta} 
 = 
\la \im [\fH , \fA] \, g_I \,\eta_{t,s}  w , \ g_I \, \eta_{t,s} w \ra + 
\sum_{j=1}^{\tN-2} \frac{c_j}{s^j}
\sum_{n,\ell,m} \langle R_n \, \eta^{(\ell)}_{t,s} w , \ \eta^{(m)}_{t,s} w \rangle + \frac{1}{s^{\tN-1}} \langle \wt R_{\tN-1} w , w \rangle 
 \end{align*}
 where each term of the form $\langle R_n \, \eta^{(\ell)}_{t,s} w , \ \eta^{(m)}_{t,s} w \rangle$ fulfills an estimate like \eqref{min4}. \\
It is finally time to use the strict Mourre estimate: by assumption (M2) we have for $s \geq 1$
\begin{align}
\la \im [\fH , \fA] \, g_I \,\eta_{t,s}  w , \ g_I \, \eta_{t,s} w \ra 
& \geq \theta \norm{g_I \, \eta_{t,s} w }_0^2   \ .
\end{align}
Using again the expansion \eqref{eta.gI} and  estimates 
\eqref{min4}, \eqref{min6} we get therefore
\begin{align}
\label{min52}
\la \im [\fH , \fA] \, g_I \,\eta_{t,s}  w , \ g_I \, \eta_{t,s} w \ra  & \geq \theta \norm{\, \eta_{t,s} g_I w }_0^2 - 
\frac{C_\tN}{s} \norm{\eta_{t,s} w}_0^2 - \frac{C_\tN}{s^{\tN-1}} \norm{w}_0^2 . 
\end{align}
This proves that the first term in the right hand side of  \eqref{min50} has a sign; we proceed from \eqref{min50} and using 
 inequality  \eqref{min52} (recall $w = g_J \psi_t)$
we  get 
\begin{equation}
\label{min53}
\la \im [\fH, \chi_{t,s}]  g_J \psi_t, \,  g_J \psi_t \ra  \leq - \frac{\theta}{s} \norm{\eta_{t,s} \, g_J \psi_t}_0^2 
 + \frac{C_\tN}{s^2}  \norm{\eta_{t,s} \, g_J \psi_t}_0^2 
 + \frac{C_\tN}{s^\tN}  \norm{\, g_J \psi_t}_0^2 \ . 
\end{equation}
We come back to the estimate \eqref{min1} of 
$\norm{(\chi_{t,s})^{\frac12}  \, g_J \psi_t}_0$.  We finally obtain,  with  $\vartheta = \frac{\theta}{2}$ and  $s \geq 1$ sufficiently large, 
\begin{align*}
\frac{d}{dt} \norm{(\chi_{t,s})^{\frac12}  \, g_J \psi_t}_0^2 
& \stackrel{\eqref{min53}}{\leq }
\frac{1}{s} \left( \vartheta - \theta \right) \norm{\eta_{t,s} \, g_J \psi_t}_0^2 
 + \frac{C_\tN}{s^2}  \norm{\eta_{t,s} \, g_J \psi_t}_0^2 
 + \frac{C_\tN}{s^\tN}  \norm{\, g_J \psi_t}_0^2
 \\
 & 
\leq  \frac{1}{s} \left(  - \frac{\theta}{2} + \frac{C_\tN}{s} \right) \norm{\eta_{t,s} \, g_J \psi_t}_0^2 
 + \frac{C_\tN}{s^\tN}  \norm{\, g_J \psi_t}_0^2  
\end{align*}
So, for $s \geq 1$ sufficiently large, the first term in the right hand side above is negative and, using also that $e^{- \im t \fH}$ is unitary and commutes with $g_J\equiv g_J(\fH)$, 
we get
\begin{equation*}
\frac{d}{dt} \norm{(\chi_{t,s})^{\frac12}  \, g_J \psi_t}_0^2  
\leq 
 \frac{C_\tN}{s^\tN}  \norm{\, g_J \psi_0}_0^2  \ . 
\end{equation*}
Integrating this inequality between $0$ and $t$   we find $\forall t>0$ 
$$
\norm{\chi^{\frac12}\Big( \frac{\fA - a - \vartheta t}{s}\Big) \, g_J \psi_t }_0^2  \leq 
\norm{\chi^{\frac12}\Big( \frac{\fA - a }{s}\Big) \, g_J \psi }_0^2 
+ \frac{C_\tN \, t }{s^\tN}  \norm{\, g_J \psi}_0^2 ,  
$$
uniformly for $a \in \R$ and $s \geq 1$ sufficiently large.
We evaluate this inequality at $a = - \frac{\vartheta}{2} t$ and $s = \sqrt{t}$,  obtaining for  $t \geq 1$ sufficiently large, the {\em minimal velocity estimate}
\begin{equation}
\label{min61}
\norm{\chi^{\frac12}\Big( \frac{\fA-\frac{\vartheta}{2}t}{\sqrt{t}}\Big) \, g_J \psi_t }_0  \leq 
\norm{\chi^{\frac12}\Big( \frac{\fA+\frac{\vartheta}{2}t}{\sqrt{t}}\Big) \, g_J \psi }_0
+ C_\tN \, t^{-\frac{\tN}{4} + \frac12}  \norm{\, g_J \psi}_0 \ . 
\end{equation}
To conclude,  take $k \in \N$ and consider $\norm{\la \fA \ra^{-k} \, g_J \,  \psi_t }_0 $. Clearly
\begin{align}
\label{min7}
\norm{\la \fA \ra^{-k} \, g_J \,  \psi_t }_0   
& \leq 
\norm{\chi^{\frac12}\Big( \frac{\fA-\frac{\vartheta}{2}t}{\sqrt{t}}\Big) \,\la \fA \ra^{-k} \, g_J \,  \psi_t}_0
\\
\label{min8}
&\quad 
+
\norm{\Big(1- \chi^{\frac12}\Big(\frac{\fA-\frac{\vartheta}{2}t}{\sqrt{t}}\Big) \,\Big) \,\la \fA \ra^{-k} \, g_J \,  \psi_t}_0
\end{align}
We estimate first \eqref{min8}. By Theorem \ref{thm:davies} (ii) we have
\begin{equation}
\label{min9}
\norm{\Big(1- \chi^{\frac12}\Big( \frac{\fA-\frac{\vartheta}{2}t}{\sqrt{t}}\Big) \,\Big)\,\la \fA \ra^{-k}}_{\cL(\cH)} \leq
\sup_{\lambda \in \R} 
\abs{\Big( 1-\chi^{\frac12}\Big(\frac{\lambda-\frac{\vartheta}{2}t}{\sqrt{t}}\Big)\Big) \,\la \lambda \ra^{-k}} \leq C_k \la t \ra^{-k} . 
\end{equation}
To prove the last inequality, use that  for $\lambda \geq \frac{\vartheta}{4}t$ then $\la \lambda \ra^{-k} \leq \la t \ra^{-k}$, whereas when $\lambda <\frac{\vartheta}{4}t$ then, being $\lambda \mapsto 1-\chi^{\frac12}(\lambda)$   monotone increasing and exponentially decaying at $-\infty$, 
$$
1-\chi^{\frac12}\Big(\frac{\lambda-\frac{\vartheta}{2}t}{\sqrt{t}}\Big) \leq 1-\chi^{\frac12}\Big(- \frac{\vartheta}{4}\sqrt{t}\Big) \leq C \left( e^{- \frac{\vartheta}{4}\sqrt{t}}\right)^{\frac12} \leq C_k \la t \ra^{-k}.
$$
Next we estimate \eqref{min7} using the minimal velocity estimate. 
As $\la \fA \ra^{-k}$ is a bounded operator,
\begin{align*}
\norm{\chi^{\frac12}\Big( \frac{\fA-\frac{\vartheta}{2}t}{\sqrt{t}}\Big) \,\la \fA \ra^{-k} \, g_J \,  \psi_t}_0
& 
\leq 
\norm{\chi^{\frac12}\Big( \frac{\fA-\frac{\vartheta}{2}t}{\sqrt{t}}\Big)  \, g_J \,  \psi_t}_0\\
& \stackrel{\eqref{min61}}{\leq  }
\norm{\chi^{\frac12}\Big( \frac{\fA+\frac{\vartheta}{2}t}{\sqrt{t}}\Big) \, g_J \psi }_0
+ C_\tN \, t^{-\frac{\tN}{4} + \frac12}  \norm{\, g_J \psi}_0
\\
& \leq
\norm{\chi^{\frac12}\Big( \frac{\fA+\frac{\vartheta}{2}t}{\sqrt{t}}\Big) \, \la \fA \ra^{-k}}_{\cL(\cH)} \, \norm{\la \fA \ra^{k} \, g_J \psi }_0
+ C_\tN \, t^{-\frac{\tN}{4} + \frac12}  \norm{\, g_J \psi}_0
\end{align*}
Again we have
\begin{equation}
\label{min10}
\norm{\chi^{\frac12}\Big( \frac{\fA+\frac{\vartheta}{2}t}{\sqrt{t}}\Big) \, \la \fA \ra^{-k}}_{\cL(\cH)}
\leq \sup_{\lambda \in \R}
\abs{\chi^{\frac12}\Big( \frac{\lambda+\frac{\vartheta}{2}t}{\sqrt{t}}\Big) \, \la \lambda \ra^{-k}} \leq C_k \la t \ra^{-k} , 
\end{equation}
since for  $\lambda \leq - \frac{\vartheta}{4} t$ one has   $\la \lambda \ra^{-k} \leq C \la t \ra^{-k}$, whereas 
 in case $\lambda > - \frac{\vartheta}{4} t$, as  $\lambda \mapsto \chi^{\frac12}(\lambda)$ is   monotone  decreasing exponentially fast at $+\infty$, one has 
$$
\chi^{\frac12}\Big(\frac{\lambda +\frac{\vartheta}{2}t}{\sqrt{t}}\Big) \leq \chi^{\frac12}\Big( \frac{\vartheta}{4}\sqrt{t}\Big) \leq C \left(e^{- \frac{\vartheta}{4}\sqrt{t}}\right)^{\frac12} \leq C_k \la t \ra^{-k}.
$$
Altogether, from \eqref{min7}, \eqref{min8} we have proved that for $t \geq 1$ sufficiently large,
\begin{align*}
\norm{\la \fA \ra^{-k} \, g_J \,  \psi_t }_0   
&\leq 
C_k \la t \ra^{-k} \norm{g_J \psi_t }_0  + 
 C_k \la t \ra^{-k} \norm{\la \fA \ra^k g_J \psi}_0 + 
 C_\tN \,  t^{-\frac{\tN}{4} + \frac12}  \norm{\, g_J \psi}_0
 \\
 & 
 \leq  C_k \la t \ra^{-k} \,  \norm{\la \fA \ra^k g_J \psi}_0 
\end{align*}
provided $\tN = 4k +2$. 
This proves the estimate \eqref{SS} for $t \geq 1$ sufficiently large, and it is also  clearly true for $t$ in any bounded interval.

\end{proof}

\bibliographystyle{plain} 

\begin{thebibliography}{10}

\bibitem{Anselone}
P. Anselone.
\newblock {\em Collectively compact operator approximation theory and
  applications to integral equations}.
\newblock Prentice-Hall, Inc., Englewood Cliffs, N. J., 1971.

\bibitem{Arbunich}
J. Arbunich, F. Pusateri, I.M. Sigal, A. Soffer.
\newblock Growth of {S}obolev norms for linear {S}chr{\"o}dinger operators.
\newblock{\em ArXiv e-print}, arXiv:2011.04570,   
  2020.

\bibitem{Bach99}
V. Bach, J.  Fr\"{o}hlich, I.M. Sigal, and A. Soffer.
\newblock Positive commutators and the spectrum of {P}auli-{F}ierz
  {H}amiltonian of atoms and molecules.
\newblock {\em Comm. Math. Phys.}, 207(3):557--587, 1999.

\bibitem{Bam16I}
D.~Bambusi.
\newblock Reducibility of 1-d {S}chr\"odinger equation with time quasiperiodic
  unbounded perturbations, {I}.
\newblock {\em Trans. Amer. Math. Soc.}, 370(3):1823--1865, 2018.

\bibitem{BGMR2}
D.~{Bambusi}, B.~{Gr{\'e}bert}, A.~{Maspero}, and D.~{Robert}.
\newblock {Growth of Sobolev norms for abstract linear Schr{\"o}dinger
  equations}.
\newblock {\em J. Eur. Math. Soc. (JEMS)}, 2020. doi: 10.4171/JEMS/1017


\bibitem{BGMR1}
D.~Bambusi, B.~Gr{\'e}bert, A.~Maspero, and D.~Robert.
\newblock Reducibility of the quantum harmonic oscillator in d-dimensions with
  polynomial time-dependent perturbation.
\newblock {\em Anal. PDE}, 11(3):775--799, 2018.

\bibitem{bambusi2020growth}
D. Bambusi, B. Langella, and R. Montalto.
\newblock Growth of {S}obolev norms for unbounded perturbations of the {L}aplacian
  on flat tori.
  \newblock{\em ArXiv e-print}, 	arXiv:2012.02654, 2020.

\bibitem{BERTI2019}
M.~Berti and A.~Maspero.
\newblock Long time dynamics of {S}chrödinger and wave equations on flat tori.
\newblock {\em J.  Diff. Eq.}, 267(2):1167 -- 1200, 2019.

\bibitem{bou99}
J.~Bourgain.
\newblock Growth of {S}obolev norms in linear {S}chr\"odinger equations with
  quasi-periodic potential.
\newblock {\em Comm. Math. Phys.}, 204(1):207--247, 1999.

\bibitem{bourgain99}
J.~Bourgain.
\newblock On growth of {S}obolev norms in linear {S}chr\"odinger equations with
  smooth time dependent potential.
\newblock {\em J. Anal. Math.}, 77:315--348, 1999.

\bibitem{chodosh}
O.~Chodosh.
\newblock Infinite matrix representations of isotropic pseudodifferential
  operators.
\newblock {\em Methods Appl. Anal.}, 18(4):351--371, 2011.

\bibitem{CdV20}
Y. Colin~de Verdi{\`e}re.
\newblock Spectral theory of pseudodifferential operators of degree 0 and an
  application to forced linear waves.
\newblock {\em Anal. PDE}, 13(5):1521--1537, 2020.

\bibitem{CdVS-R}
Y. Colin~de Verdi{\`e}re and Laure Saint-Raymond.
\newblock Attractors for two-dimensional waves with homogeneous {H}amiltonians
  of degree 0.
\newblock {\em Comm. Pure Appl. Math.}, 73(2):421--462, 2020.

\bibitem{CFKS}
H. Cycon, R. Froese, W.~Kirsch, and B.~Simon.
\newblock {\em Schr\"{o}dinger operators with application to quantum mechanics
  and global geometry}.
\newblock Texts and Monographs in Physics. Springer-Verlag, Berlin, study
  edition, 1987.

\bibitem{davies}
E. Davies.
\newblock The functional calculus.
\newblock {\em J. London Math. Soc. (2)}, 52(1):166--176, 1995.

\bibitem{del2}
J.-M. Delort.
\newblock Growth of {S}obolev norms of solutions of linear {S}chr\"odinger
  equations on some compact manifolds.
\newblock {\em Int. Math. Res. Not. IMRN}, (12):2305--2328, 2010.

\bibitem{del}
J.-M. Delort.
\newblock Growth of {S}obolev norms for solutions of time dependent
  {S}chr\"odinger operators with harmonic oscillator potential.
\newblock {\em Comm. Partial Differential Equations}, 39(1):1--33, 2014.

\bibitem{DG}
J. Derezi\'{n}ski and C. G\'{e}rard.
\newblock {\em Scattering theory of classical and quantum {$N$}-particle
  systems}.
\newblock Texts and Monographs in Physics. Springer-Verlag, Berlin, 1997.

\bibitem{DiSj}
M.~Dimassi and J.~Sjostrand.
\newblock {\em Spectral Asymptotics in the Semi-Classical Limit}.
\newblock London Mathematical Society Lecture Note Series. Cambridge University
  Press, 1999.

\bibitem{DyatlovZworski}
S. Dyatlov and M. Zworski. 
\newblock{ Microlocal analysis of forced waves.}
\newblock {\em  Pure Appl. Anal. }, 1(3): 359--384, 2019. 

\bibitem{duclos}
P.~Duclos, O.~Lev, and P.~S\v tov\'\i \v cek.
\newblock On the energy growth of some periodically driven quantum systems with
  shrinking gaps in the spectrum.
\newblock {\em J. Stat. Phys.}, 130(1):169--193, 2008.


\bibitem{FaouRaph}
E. Faou and P. Raphael.
\newblock On weakly turbulent solutions to the perturbed linear harmonic
  oscillator.
\newblock{\em ArXiv e-print}, arXiv:2006.08206,   
  2020.

\bibitem{GerSig}
C.~G\'{e}rard and I.~M. Sigal.
\newblock Space-time picture of semiclassical resonances.
\newblock {\em Comm. Math. Phys.}, 145(2):281--328, 1992.

\bibitem{GNRS}
E. Grenier, T. Nguyen, F. Rousset, and A. Soffer.
\newblock Linear inviscid damping and enhanced viscous dissipation of shear
  flows by using the conjugate operator method.
\newblock {\em J. Funct. Anal.}, 278(3):108339, 27, 2020.

\bibitem{HausMaspero}
E. Haus and  A. Maspero.
\newblock{ Growth of Sobolev norms in time dependent semiclassical anharmonic oscillators}. 
\newblock{\em J. Funct. Anal.} , 278(2), 108316, 2020.

\bibitem{HS}
B.~Helffer and J.~Sj\"{o}strand.
\newblock \'{E}quation de {S}chr\"{o}dinger avec champ magn\'{e}tique et
  \'{e}quation de {H}arper.
\newblock In {\em Schr\"{o}dinger operators}, volume 345
  of {\em Lecture Notes in Phys.},  118--197. Springer, Berlin, 1989.



\bibitem{ho2}
L.~H\"{o}rmander.
\newblock The {W}eyl calculus of pseudodifferential operators.
\newblock {\em Comm. Pure Appl. Math.}, 32(3):360--444, 1979.

\bibitem{ho}
L.~H{\"o}rmander.
\newblock {\em The analysis of linear partial differential operators I-IV}.
\newblock Grundlehren der mathematischen Wissenschaften 256. Springer-Verlag,
  1985.

\bibitem{HuSi}
W.~Hunziker and I.~M. Sigal.
\newblock Time-dependent scattering theory of n-body quantum systems.
\newblock {\em Reviews in Mathematical Physics}, 12(08):1033--1084, 2000.

\bibitem{HunSigSof}
W.~Hunziker, I.~M. Sigal, and A.~Soffer.
\newblock Minimal escape velocities.
\newblock {\em Comm. Partial Differential Equations}, 24(11-12):2279--2295,
  1999.

\bibitem{JenMouPer}
A. Jensen, \'{E}. Mourre, and P. Perry.
\newblock Multiple commutator estimates and resolvent smoothness in quantum
  scattering theory.
\newblock {\em Ann. Inst. H. Poincar\'{e} Phys. Th\'{e}or.}, 41(2):207--225,
  1984.

\bibitem{Kato}
T. Kato.
\newblock {\em Perturbation theory for linear operators}.
\newblock Classics in Mathematics. Springer-Verlag, Berlin, 1995.


\bibitem{MaRo}
A.~Maspero and D.~Robert.
\newblock On time dependent {S}chr{\"o}dinger equations: {G}lobal
  well-posedness and growth of {S}obolev norms.
\newblock {\em J.  Fun. Anal.}, 273(2):721 -- 781, 2017.

\bibitem{Mas19}
A. Maspero.
\newblock Lower bounds on the growth of {S}obolev norms in some linear time
  dependent {S}chr\"{o}dinger equations.
\newblock {\em Math. Res. Lett.}, 26(4):1197--1215, 2019.

\bibitem{MONTALTO2019}
R. Montalto.
\newblock Growth of {S}obolev norms for time dependent periodic {S}chrödinger
  equations with sublinear dispersion.
\newblock {\em J. Diff.  Eq.}, 266(8):4953 -- 4996, 2019.

\bibitem{Mourre}
E.~Mourre.
\newblock Absence of singular continuous spectrum for certain selfadjoint
  operators.
\newblock {\em Comm. Math. Phys.}, 78(3):391--408, 1980/81.

\bibitem{nen}
G.~Nenciu.
\newblock Adiabatic theory: stability of systems with increasing gaps.
\newblock {\em Annales de l'I. H. P}, 67-4:411--424, 1997.

\bibitem{ReedSimon}
M. Reed and B. Simon.
\newblock {\em Methods of modern mathematical physics. {I}}.
\newblock Academic Press, Inc., New
  York, second edition, 1980.
\newblock Functional analysis.

\bibitem{SarVai}
J. Saranen and G. Vainikko.
\newblock {\em Periodic integral and pseudodifferential equations with
  numerical approximation}.
\newblock Springer Monographs in Mathematics. Springer-Verlag, Berlin, 2002.

\bibitem{Shubin}
M. Shubin.
\newblock {\em Pseudodifferential operators and spectral theory}.
\newblock Springer-Verlag, Berlin, second edition, 2001.

\bibitem{SS}
I.M.~Sigal and A.~Soffer.
\newblock Local decay and velocity bounds for quantum propagation.
\newblock {\em preprint (Princeton)}, 1988.
\href{}{http://www.math.toronto.edu/sigal/publications/SigSofVelBnd.pdf}

\bibitem{Ski}
E. Skibsted.
\newblock Propagation estimates for {$N$}-body {S}chroedinger operators.
\newblock {\em Comm. Math. Phys.}, 142(1):67--98, 1991.

\bibitem{sogge}
C. Sogge.
\newblock {\em Hangzhou Lectures on Eigenfunctions of the Laplacian}.
\newblock Princeton University Press, 2014.

\bibitem{Thomann}
L. Thomann.
\newblock Growth of {S}obolev norms for linear {S}chr{\"o}dinger operators.
\newblock{\em ArXiv e-print}, arXiv:2006.02674,   
  2020.


\end{thebibliography}

\end{document}